\documentclass[11pt]{amsart}
\usepackage{amsthm,amssymb,mathrsfs,mathtools,empheq}

\usepackage[T1]{fontenc}

\usepackage[notref,notcite,final]{showkeys}
\mathtoolsset{showonlyrefs}
\usepackage{fullpage,color}
\usepackage{esint}

\allowdisplaybreaks

\newtheorem{mainthm}{Theorem}
\newtheorem{theorem}{Theorem}[section]
\newtheorem*{theorem*}{Theorem}

\newtheorem{lemma}[theorem]{Lemma}
\newtheorem{proposition}[theorem]{Proposition}
\newtheorem*{proposition*}{Proposition}

\newtheorem*{conjecture*}{Conjecture}

\theoremstyle{definition}

\newtheorem{remark}[theorem]{Remark}

\numberwithin{equation}{section}

\def\bN {\mathbb{N}}

\def\bR {\mathbb{R}}

\def\cW {\mathcal{W}}

\def\cY {\mathcal{Y}}
\def\cZ {\mathcal{Z}}

\def\scrC{\mathscr{C}}

\def\grad {{\nabla}}

\def\la {\langle}
\def\ra {\rangle}

\def\wto {{\rightharpoonup}}

\newcommand{\tx}[1]{\mathrm{#1}}

\newcommand{\wt}[1]{\widetilde{#1}}
\newcommand{\bs}[1]{\boldsymbol{#1}}

\newcommand{\ud}{\mathrm{\,d}}
\newcommand{\vd}{\mathrm{d}}

\newcommand{\dd}[1]{{\frac{\vd}{\vd{#1}}}}
\newcommand{\pd}[1]{{\frac{\partial}{\partial{#1}}}}

\newcommand{\enorm}{{\dot H^1\times L^2}}

\newcommand{\uln}[1]{{\underline{ #1 }}}

\newcommand{\ulstar}{{u\lin\!\!^*}}
\newcommand{\Mod}{\mathrm{Mod}}
\newcommand{\Int}{\mathrm{Int}}
\newcommand{\lin}{{_{\textsc l}}}

\title{Construction of type II blow-up solutions for the energy-critical wave equation in dimension 5}
\author{Jacek Jendrej}
\address{Department of Mathematics, University of Chicago, 5734 S. University Avenue, Chicago, IL 60637}
\email{jacek@math.uchicago.edu}
\keywords{blow-up; critical nonlinearity; soliton}

\begin{document}
%
\begin{abstract}
We consider the semilinear wave equation with focusing energy-critical nonlinearity in space dimension $N=5$
\begin{equation*}
  \partial_{tt}u = \Delta u + |u|^{4/3}u,
\end{equation*}
with radial data. It is known \cite{DKM2} that a solution $(u, \partial_t u)$ which blows up at $t = 0$
in a neighborhood (in the energy norm) of the family of solitons $W_\lambda$, decomposes in the energy space as
\begin{equation*}
  (u(t), \partial_t u(t)) = (W_{\lambda(t)} + u^*_0, u^*_1) + o(1),
\end{equation*}
where $\lim_{t\to 0}\lambda(t)/t = 0$ and $(u^*_0, u^*_1) \in \dot H^1\times L^2$.
We construct a blow-up solution of this type
such that the asymptotic profile $(u^*_0, u^*_1)$ is any pair of sufficiently regular functions with $u_0^*(0) > 0$.
For these solutions the concentration rate is $\lambda(t) \sim t^4$.
We also provide examples of solutions with concentration rate $\lambda(t) \sim t^{\nu + 1}$ for $\nu > 8$,
related to the behaviour of the asymptotic profile near the origin.
\end{abstract}
\maketitle{}

\section{Introduction}
\subsection{General setting}
We are interested in the problem of constructing type II blow-up solutions for the energy-critical wave equation in space dimension $N=5$:
\begin{equation*}
  \partial_{tt}u(t, x) = \Delta u(t, x) + |u(t, x)|^{4/3}u(t, x),\qquad (t, x) \in \bR\times \bR^5.
\end{equation*}
Denote $f(u) := |u|^{4/3}u$. It will be convenient to write the wave equation as a first-order in time system:
\begin{empheq}{equation}
  \label{eq:nlw}
  \left\lbrace
  \begin{aligned}
  \pd t
  \begin{pmatrix}
    u \\ \partial_t u
  \end{pmatrix}
  &=
  \begin{pmatrix}
    \partial_t u \\
    \Delta u + f(u)
  \end{pmatrix}, \\
  \begin{pmatrix}
    u(t_0) \\ \partial_t u(t_0)
  \end{pmatrix}
  &=
  \begin{pmatrix}
    u_0 \\ u_1
  \end{pmatrix} \in \dot H^1(\bR^5)\times L^2(\bR^5).
\end{aligned}\right.
\tag{NLW}
\end{empheq}
This equation is locally well-posed in the energy space $\dot H^1\times L^2$ (see for example \cite{KeMe08} and the references therein).
In particular, for any initial data $(u_0, u_1)$ there exists a maximal interval of existence $(T_-, T_+)$, $-\infty \leq T_- < t_0 < T_+ \leq +\infty$,
and a unique solution $(u, \partial_t u) \in C((T_-, T_+); \enorm) \cap L_\tx{loc}^{\frac{7}{3}}((T_-, T_+); L^\frac{14}{3}(\bR^5))$.
This solution conserves the energy:
\begin{equation*}
  E(u(t), \partial_t u(t)) = \frac 12 \int|\partial_t u|^2\ud x + \frac 12\int|\grad u|^2\ud x - \int F(u)\ud x = E(u_0, u_1),
\end{equation*}
where $F(u) = \int f(u)\ud u = \frac{3}{10}|u|^{10/3}$
(notice that $\int F(u)\ud x$ is finite by the Sobolev Embedding Theorem).

For a function $v: \bR^5 \to \bR$ and $\lambda > 0$, we denote
\begin{equation*}
  v_\lambda(x) := \frac{1}{\lambda^{3/2}} v\big(\frac{x}{\lambda}\big), \qquad v_\uln\lambda(x) := \frac{1}{\lambda^{5/2}} v\big(\frac{x}{\lambda}\big).
\end{equation*}
A change of variables shows that
\begin{equation*}
  E\big((u_0)_\lambda, (u_1)_\uln\lambda\big) = E(u_0, u_1).
\end{equation*}
Equation \eqref{eq:nlw} is invariant under the same scaling. If $(u, \partial_t u)$ is a solution of \eqref{eq:nlw} and $\lambda > 0$, then
$$
t \mapsto \Big(u\big(\frac{t}{\lambda} + t_0\big)_\lambda,\,\partial_t u\big(\frac{t}{\lambda} + t_0\big)_\uln\lambda\Big)
$$ is also a solution
with initial data $\big((u_0)_\lambda, (u_1)_\uln\lambda\big)$ at time $t = 0$.
This is why equation \eqref{eq:nlw} is called \emph{energy-critical}.

We introduce also the infinitesimal generators of scale change:
\begin{equation*}
  \begin{aligned}
  \Lambda v &:= -\pd \lambda v_\lambda \Big\vert_{\lambda = 1} = \big(\frac 32 + x\cdot\grad\big)v, \\
  \Lambda_0 v &:= -\pd \lambda v_\uln\lambda \Big\vert_{\lambda = 1} = \big(\frac 52 + x\cdot\grad\big)v.
\end{aligned}
\end{equation*}

A fundamental object in the study of \eqref{eq:nlw} is the family of solutions $(u, \partial_t u) = (W_\lambda, 0)$, where
\begin{equation*}
  W(x) = \Big(1 + \frac{|x|^2}{15}\Big)^{-3/2}.
\end{equation*}
The functions $W_\lambda$ are called \emph{ground states}. 
In this paper we are interested in radial solutions $(u, \partial_t u)$ of \eqref{eq:nlw} such that $\inf_\lambda\|(u - W_\lambda, \partial_t u)\|_\enorm$
remains small for $T_- < t \leq t_0$. In the case $N = 3$ it was proved by Krieger, Nakanishi and Schlag \cite{KrNaSc15} that such solutions form a codimension one manifold
in a neighborhood of the family $\{W_\lambda\}$. This is expected to hold also for $N = 5$.
The asymptotic behavior of such (not necessarily radial) solutions as $t \to T_-$ was described by Duyckaerts, Kenig and Merle in \cite{DKM2}, both in the case $T_- = -\infty$ and $T_- > -\infty$.
In the second case, which is relevant for us, they obtain the following result.
\begin{theorem*}\cite[Theorem 2]{DKM2}
  Let $(u, \partial_t u)$ be a solution of \eqref{eq:nlw} such that $T_- = 0$ and $\inf_\lambda\|(u - W_\lambda, \partial_t u)\|_\enorm$
  remains small for $T_- < t \leq T_0$. Then there exists a $C^0$ function $\lambda(t): (0, T_0) \to (0, +\infty)$,
  such that
  \begin{equation}
    \label{eq:1bubble}
    \lim_{t\to 0^+} \big(u(t) - W_{\lambda(t)}, \partial_t u(t)\big) = (u_0^*, u_1^*) \in \enorm,
  \end{equation}
  and the convergence is strong in $\enorm$. In addition, $\lambda(t) \ll t$ as $t \to 0^+$.
\end{theorem*}
In this context, $W_\lambda$ is called the \emph{bubble of energy} and $(u_0^*, u_1^*)$ is called the \emph{asymptotic profile}.

Solutions of this type were first constructed by Krieger, Schlag and Tataru \cite{KrScTa09} in space dimension $N = 3$,
where it is shown that for any $\nu > 1/2$ there exists a solution such that the concentration speed is $\lambda(t) \sim t^{1+\nu}$
(later Krieger and Schlag \cite{KrSc14} improved this to $\nu > 0$).
Similar results where obtained for energy-critical wave maps by the same authors \cite{KrScTa08},
for energy-critical NLS in dimension $N = 3$ by Ortoleva and Perelman \cite{OrPe13}
and for energy-critical Schr\"odinger maps by Perelman \cite{Perelman14}.
Using a different approach, Hillairet and Rapha\"el \cite{HiRa12} obtained $C^\infty$ blow-up solutions for energy-critical wave equation in dimension $N = 4$
with blow-up rate $\lambda(t) = t\exp\big({-}\sqrt{-\log t}(1+o(1))\big)$.
Collot \cite{Collot14p} obtained a related result for supercritical wave equation in large dimension.

It follows from the classification of solutions with energy $E(W)$ by Duyckaerts and Merle \cite{DM08} that necessarily $(u_0^*, u_1^*) \neq 0$.
In other words, we have non-existence of minimal energy blow-up solutions. An analogous result is true also for
energy-critical wave maps, energy-critical Schr\"odinger maps and energy-critical NLS.

This is in contrast with the $L^2$-critical NLS where the conformal invariance produces explicit solutions
concentrating a bubble of mass and tending weakly to 0 at blow-up. Existence of blow-up solutions with a non-zero smooth asymptotic profile
was first observed by Bourgain and Wang \cite{BoWa97}. Blow-up solutions close to the ground state in the case of $L^2$-critical NLS were extensively
studied in a series of papers by Merle and Rapha\"el. They examined in particular the relationship between regularity of the asymptotic profile
and the blow-up speed. One can consult a survey \cite{MMRS14} for an account of these results in a proper perspective
and a presentation of some recent developments in the case of $L^2$-critical gKdV.

\subsection{Main results}
The aim of this paper is to construct solutions which blow up by concentration of one bubble of energy in space dimension $N = 5$.
Our approach differs substantially from \cite{KrScTa09} in that it produces a blow-up solution with a given asymptotic profile.
This profile is seen as a source term which permits concentration of the bubble.
This point of view is close to a recent construction by Martel, Merle and~Rapha\"el \cite{MMR15-3} in the case of $L^2$-critical gKdV.

Denote $X^s := \dot H^{s+1}\cap \dot H^1$. We prove the following two results.
\begin{mainthm}
  \label{thm:non-deg}
  Let $(u^*_0, u^*_1) \in X^4\times H^4$ be any pair of radial functions with $u^*_0(0) > 0$.
  Let $(u^*(t), \partial_t u^*(t))$ be the solution of \eqref{eq:nlw} with the initial data $(u^*(0), \partial_t u^*(0)) = (u^*_0, u^*_1)$.
  There exists a solution $(u, \partial_t u)$
  of \eqref{eq:nlw} defined on a time interval $(0, T_0)$ and a $C^1$ function $\lambda(t): (0, T_0) \to (0, +\infty)$ such that
  \begin{equation}
    \label{eq:blowup-1}
    \|(u(t) - W_{\lambda(t)} - u^*(t),\ \partial_t u(t) +\lambda_t(t) (\Lambda W)_{\uln{\lambda(t)}}- \partial_t u^*(t))\|_\enorm = O(t^{9/2})\qquad \text {as }t \to 0^+,
  \end{equation}
  and $\lambda(t) = \big(\frac{32}{315\pi}\big)^2\big(u^*(0, 0)\big)^2 t^4 + o(t^4)$.
\end{mainthm}
\begin{mainthm}
  \label{thm:deg}
  Let $\nu > 8$. There exists a solution $(u, \partial_t u)$ of \eqref{eq:nlw}
  defined on the time interval $(0, T_0)$  such that
  \begin{equation}
    \label{eq:blowup-2}
    \lim_{t \to 0^+}\big\|\big(u(t) - W_{\lambda(t)} - u^*_0,\ \partial_t u(t) - u^*_1\big)\big\|_\enorm = 0,
  \end{equation}
  where $\lambda(t) = t^{\nu+1}$, and $(u^*_0, u^*_1)$ is an explicit radial $C^2$ function.
\end{mainthm}
We will refer to the situation of Theorem 
\ref{thm:non-deg} as the \emph{non-degenerate case} and to the situation of Theorem 
\ref{thm:deg} as the \emph{degenerate case}.
Note that in Theorem 
\ref{thm:non-deg} we allow any regular $(u_0^*, u_1^*)$ with $u_0^*(0) > 0$.
Our result might be seen as a first step in a possible classification of all blow-up solutions with a non-degenerate asymptotic profile.
Theorem 
\ref{thm:deg} demonstrates how the asymptotic behaviour of $(u^*_0, u^*_1)$ at $x = 0$ influences the blow-up speed.
The condition $\nu > 8$ is imposed by our method. It could be improved at the cost of some technical details,
but we are far from obtaining the whole range $\nu > 0$ as in \cite{KrSc14} for $N = 3$.

Let us mention that radiality is only a simplifying assumption. A similar construction should be possible also for non-radial $(u^*_0, u^*_1)$.

In Theorem 
\ref{thm:deg}, the function $u^*_0$ is given explicitely by \eqref{eq:ustar-deg} and $u^*_1 = 0$.
It follows from our proof that there exists a $C^1$ function $\wt\lambda(t): (0, T_0) \to (0, +\infty)$
such that $\wt\lambda(t) = t^{\nu+1} + o(t^{\nu+1})$ and the solution $(u, \partial_t u)$ satisfies
\begin{equation}
  \label{eq:blowup-2prim}
  \big\|\big(u(t) - W_{\wt\lambda(t)} - u^*(t),\ \partial_t u(t) + \wt\lambda_t(t) (\Lambda W)_{\uln{\wt\lambda(t)}}- \partial_t u^*(t)\big)\big\|_\enorm
  = O(t^{\frac 76 \nu - \frac 43}).
\end{equation}

\subsection{Structure of the proof}
\label{ssec:structure}

The first part of the proof consists in obtaining an accurate enough approximate solution, or \emph{ansatz}, $(\varphi_0(t), \varphi_1(t))$.
This ansatz depends on the modulation parameter $\lambda(t)$ and an auxiliary parameter $b(t)$, and is close in the energy space $\dot H^1 \times L^2$
to $(W_{\lambda(t)} + u_0^*, u_1^*)$.

A natural way to define such an ansatz is to apply a version of the separation of variables method.
This leads to a solvability condition yielding a system of differential equations for $\lambda(t)$ and $b(t)$.
The derivation of these equations at a formal level is presented in Section~\ref{sec:formal}.
It explains the relation between the asymptotic behaviour of $(u_0^*, u_1^*)$ and the blow-up speed,
as well as the relevance of the condition $u^*_0(0) > 0$.

In Section~\ref{sec:non-deg} we give a precise definition of $(\varphi_0(t), \varphi_1(t))$ in the non-degenerate case.
We prove bounds on the error of this approximate solution under some assumptions on the parameters $\lambda$ and $b$
which will be satisfied when we will use these bounds in the second part of the proof.
In particular, we assume that $\lambda(t)$ and $b(t)$ approximately solve the differential equations
found in the preceding section.

In Section~\ref{sec:deg} we choose $(u^*_0(0), u^*_1(0))$ such that the differential equations for $\lambda$ and $b$
obtained in Section~\ref{sec:formal} lead to $\lambda(t) \sim t^{1+\nu}$. Next, we repeat the procedure from Section~\ref{sec:non-deg}.

In Section~\ref{sec:everr} we analyze solutions $(u, \partial_t u)$ of \eqref{eq:nlw} such that $(u, \partial_t u) = (\varphi_0, \varphi_1) + (\varepsilon_0, \varepsilon_1)$
with $\|(\varepsilon_0, \varepsilon_1)\|_\enorm$ small. In order to control $\|(\varepsilon_0, \varepsilon_1)\|_\enorm$,
we use the energy functional $I(t)$, which is simply the energy of $(u, \partial_t u)$ with the terms of order $0$ and $1$ in $(\varepsilon_0, \varepsilon_1)$ removed.
Using classical variational properties of $W$ one can show that, modulo eigendirections of the flow (which can be easily controlled),
we have $\|(\varepsilon_0, \varepsilon_1)\|_\enorm^2 \lesssim I(t)$.
We need to correct $I(t)$ with a virial-type functional, which yields a mixed energy-virial functional $H(t)$.
The correction is negligible with respect to $\|(\varepsilon_0, \varepsilon_1)\|_\enorm^2$, hence we can still
bound $\|(\varepsilon_0, \varepsilon_1)\|_\enorm^2$ by $H(t)$.
However, the virial part is significant when we compute the time derivative of $H(t)$.
We will essentially prove that
$$
H'(t) \leq \frac ct \|(\varepsilon_0, \varepsilon_1)\|_\enorm^2 + C_1t^\gamma\|(\varepsilon_0, \varepsilon_1)\|_\enorm,
$$
where $c$ is a small constant, $C_1$ is a large constant and $t^\gamma$ is the size of the error of the approximate solution.
The above inequality yields $\|(\varepsilon_0, \varepsilon_1)\|_\enorm \lesssim t^{\gamma+1}$
by a straightforward continuity argument.

The reason for using a virial correction can be understood by considering the following quadratic term with a large potential concentrated at scale $\lambda$:
$$
\int_{\bR^5} f'(W_\lambda)\varepsilon_0^2 \ud x = \int_{\bR^5} \frac{1}{\lambda^2}V\Big(\frac{\cdot}{\lambda}\Big)\varepsilon_0^2\ud x.
$$
If we differentiate the potential directly with respect to time, we obtain the term
$$
\int_{\bR^5}{-}\frac{\lambda_t}{\lambda^3}(2V+ x\cdot \grad V)\Big(\frac{\cdot}{\lambda}\Big)\varepsilon_0^2\ud x.
$$
If $\lambda \sim t^{1+\nu}$, then $\big|\frac{\lambda_t}{\lambda}\big| \sim \frac 1t$,
hence there is no hope of closing a bootstrap argument.
Instead, we can differentiate at scale $\lambda$ and then scale back:
$$
\begin{aligned}
  \dd t\int_{\bR^5}\frac{1}{\lambda^2}V\Big(\frac{\cdot}{\lambda}\Big)\varepsilon_0^2 \ud x &= \dd t\int_{\bR^5} V(\cdot)(\varepsilon_0)_{1/\lambda}^2\ud x \\
  &= 2\int_{\bR^5} V(\cdot)(\varepsilon_0)_{1/\lambda}\Big(\frac{\lambda_t}{\lambda}(\Lambda \varepsilon_0)_{1/\lambda} + (\partial_t \varepsilon_0)_{1/\lambda}\Big)\ud x \\
  &= 2\int_{\bR^5}\frac{1}{\lambda^2}V\Big(\frac{\cdot}{\lambda}\Big)\varepsilon_0\Big(\frac{\lambda_t}{\lambda}\Lambda \varepsilon_0 + \partial_t \varepsilon_0\Big)\ud x
\end{aligned}
$$
which is how the virial correction appears.

The method of using a mixed energy-virial functional to control the error term was introduced by Rapha\"el and Szeftel \cite{RaSz11}
for a construction of minimal mass blow-up solutions for inhomogeneous $L^2$-critical NLS.

In Section~\ref{sec:shooting} we follow a well-known compactness argument introduced by Merle \cite{Merle90}
and used by several authors starting with the work of Martel \cite{Martel05} for constructions of multi-solitons.
We take a decreasing sequence $t_n \to 0^+$ and we define $(u_n, \partial_t u_n)$ as the solution of \eqref{eq:nlw}
such that $(u_n(t_n), \partial_t u_n(t_n))$ is close to the approximate solution at time $t = t_n$.
By a continuity argument, we obtain that there exists $T_0 > 0$, independent of $n$, such that $\|(\varepsilon_0, \varepsilon_1)\|_\enorm \leq Ct^{\gamma+1}$
for $t \in [t_n, T_0]$, with $C$ independent of $n$.
It turns out that this bound on $(\varepsilon_0, \varepsilon_1)$ is sufficient to prove that the modulation parameters $\lambda$ and $b$
are close to the formally predicted values from Section~\ref{sec:formal}.
The reason for this is that the solution of the differential equations for $\lambda$ and $b$ obtained in Section~\ref{sec:formal} is stable.
In this step it is important to have $\gamma$ sufficiently large, in other words to work with a sufficiently accurate approximate solution.

Note that the exponential instability of $W_\lambda$ causes an additional difficulty in the argument.
We use the shooting method to eliminate the unstable mode.
The desired blow-up solution $(u, \partial_t u)$ is obtained by taking a weak limit $(u_0, u_1)$ of a subsequence of $(u_n(T_0), \partial_t u_n(T_0))$
and solving \eqref{eq:nlw} with the initial data $u(T_0) = u_0$ and $\partial_t u(T_0) = u_1$.

In Appendix~\ref{sec:weak} we prove sequential weak continuity of the dynamical system \eqref{eq:nlw} under some natural (non-optimal) condition,
which is an adaptation of an analogous result of Bahouri and G\'erard in the defocusing case \cite[Corollary 1]{BaGe99}.
This result is required in the last step of the proof.

In Appendix \ref{sec:cauchy} we provide for the reader's convenience some well-known estimates of the $X^1 \times H^1$ norm of solutions of \eqref{eq:nlw}.
The persistence of $X^1\times H^1$ regularity is used in Section~\ref{sec:everr}. The energy estimates are used in Section~\ref{sec:deg}.
They are non-optimal, but sufficient for our purposes. We prove also propagation of regularity in a neighbourhood of the origin in the non-degenerate case,
which is used in Section~\ref{sec:non-deg}.

\subsection{Acknowledgements}
This paper was prepared as a part of my PhD under supervision of Y.~Martel and F.~Merle, at \'Ecole polytechnique in Palaiseau, France.
I was partially supported by the ERC grant $291214$ BLOWDISOL.

\subsection{Notation}
\label{ssec:notation}
For $v, w \in L^2$ we denote
\begin{equation*}
  \la v, w\ra := \int_{\bR^5} v\cdot w\ud x.
\end{equation*}
We use the same notation for the duality pairing when $v \in \dot H^{-s}$ and $w \in \dot H^s$.

Linearizing $-\Delta V - f(V)$ around $V = W_\lambda$ we obtain a self-adjoint operator
\begin{equation*}
  L_\lambda h := -\Delta h - f'(W_\lambda)h.
\end{equation*}
Differentiating $-\Delta W_\lambda - f(W_\lambda) = 0$ with respect to $\lambda$ we find $$L_\lambda (\Lambda W)_\lambda = 0.$$
We denote $L := L_1 = -\Delta - f'(W)$.

We will also use the notation $\bs v(t) := (v(t), \partial_t v(t))$.

We denote $\cZ$ a fixed radial $C_0^\infty$ function such that $\la \Lambda W, \cZ \ra > 0$.

Finally, $\chi$ is a fixed standard $C^\infty$ cut-off function, that is $\chi(r) = 1$ for $r \leq 1$, $\chi(r) = 0$ for $r \geq 2$, $\chi'(r) \leq 0$.

\section{Formal picture and construction of blow-up profiles}
\label{sec:formal}
\subsection{Inverting the operator $L$}
We define
\begin{equation}\label{eq:c}
  \kappa := -\frac{\la\Lambda W, f'(W)\ra}{\la \Lambda W, \Lambda W\ra} = \frac{128}{105\pi}.
\end{equation}

\begin{proposition}
  \label{prop:fredholm}
  There exist radial functions $A, B \in C^\infty(\bR^5)$ such that
  \begin{equation}
    \label{eq:AB}
    LA = \kappa \Lambda W + f'(W), \qquad LB = -\Lambda_0 \Lambda W.
  \end{equation}
  In addition, $A(r) \sim r^{-1}$, $A'(r) \sim r^{-2}$, $A''(r) \sim r^{-3}$ and $B(r) \sim r^{-1}$, $B'(r) \sim r^{-2}$, $B''(r) \sim r^{-3}$ as $r \to +\infty$.
\end{proposition}
\begin{proof}
In the proof we will use some standard facts from the theory of Sturm-Liouville equations, see for example \cite[Chapter 5]{teschl12}.

Solving equation \eqref{eq:AB} is equivalent to solving the following ODE:
\begin{equation}\label{eq:sl}
  -(p(r)y')'+q(r)y=g(r),
\end{equation}
with $r\in(0,+\infty)$, $p(r)=r^4$, $q(r)=-r^4f'(W)$ and $g(r) = g_A(r) = r^4(\kappa\Lambda W(r) + f'(W(r)))$ or $g(r) = g_B(r) = -r^4\Lambda_0\Lambda W(r)$.
Notice that $|g(r)| \lesssim r^4$ for small $r$.

We know that $\Lambda W(r)$ is a solution of \eqref{eq:sl} with $g(r)=0$. Let $\Gamma(r)$ be a second solution normalized in such a way that
\begin{equation}
  \label{eq:wronskian}
  \cW(\Lambda W, \Gamma) = r^4(\Lambda W\cdot \Gamma' - (\Lambda W)'\cdot \Gamma) = 1
\end{equation}
($\cW$ is the modified wronskian, in particular its value is independent of $r$).

Take $r_1 < \sqrt{15}$, $r_2 > \sqrt{15}$ (recall that $r = \sqrt{15}$ is the unique point where $\Lambda W$ vanishes) and define
\begin{equation*}
  \begin{aligned}
    y_1(r) &:= \Lambda W(r)\cdot \int_{r_1}^r\frac{\vd s}{s^4(\Lambda W(s))^2},\qquad\text{for }r < \sqrt{15},  \\
    y_2(r) &:= \Lambda W(r)\cdot \int_{r_2}^r\frac{\vd s}{s^4(\Lambda W(s))^2},\qquad\text{for }r > \sqrt {15}.
  \end{aligned}
\end{equation*}
It can be easily checked that $y_1$ and $y_2$ are solutions of the homogeneous equation and verify $\cW(\Lambda W, y_1) = \cW(\Lambda W, y_2) = 1$.
Hence, we have $y_j = a_j \Lambda W + \Gamma$ for some scalar coefficients $a_1, a_2$.
Directly from the formulas defining $y_1$ and $y_2$ we obtain the asymptotic behaviour of $y_1$ as $r \to 0^+$ and of $y_2$ as $r \to \infty$
:\begin{equation*}
  \begin{aligned}
    y_1(r) &\sim -\int_r^{r_1}\frac{\vd s}{s^4} \sim -\frac{1}{r^3},\qquad r\to 0^+,\\
    y_2(r) &\sim \frac{-1}{r^3}\int_{r_2}^r\frac{\vd s}{s^4\cdot s^{-6}} \sim -1,\qquad r\to +\infty.
  \end{aligned}
\end{equation*}
As adding a constant multiple of $\Lambda W$ does not change these asymptotics, we obtain that $\Gamma(r) \sim -r^{-3}$ as $r \to 0^+$ and $\Gamma(r) \sim -1$ as $r\to+\infty$.
From the relation $\cW(\Lambda W, \Gamma) = 1$ we get
\begin{equation*}
  \Gamma'=\frac{r^{-4}+(\Lambda W)'\cdot \Gamma}{\Lambda W},
\end{equation*}
which immediately gives $\Gamma'(r) \sim r^{-4}$ as $r \to 0$ and $\Gamma'(r) \sim \pm r^{-1}$ as $r\to +\infty$ (it can be checked that the sign is $"+"$,
but we will not use this fact).

For $r_0, r \in (0, +\infty)$ we define
\begin{equation}
  \label{eq:fondam}
  s(r, r_0) := \Lambda W(r_0) \Gamma(r) - \Gamma(r_0) \Lambda W(r).
\end{equation}
We see that $s(r_0, r_0) = 0$ and $r_0^4\dd r s(r, r_0)\vert_{r=r_0}=1$, which means that $s(r, r_0)$ is the second fundamental solution of \eqref{eq:sl}.
Now using the Duhamel formula we obtain a solution of the non-homogeneous equation \eqref{eq:sl}:
\begin{equation}
  \label{eq:y}
  \begin{aligned}
  A(r) &= \int_0^r s(r, r')g_A(r')\ud r', \\
  B(r) &= \int_0^r s(r, r')g_B(r')\ud r'.
  \end{aligned}
\end{equation}

Fix $r > 0$ and let $|h| \leq\frac 12 r$. In the estimates which follow, all the constants may depend on $r$. We have
\begin{equation*}
  \begin{aligned}
  &\bigl|\frac{A(r+h)-A(r)}{h}-\int_0^r\dd r s(r, r')g_A(r')\ud r'\bigr| \\
	&\leq \int_0^r\bigl|\frac{s(r+h, r')-s(r, r')}{h}-\dd rs(r, r')\bigr|\cdot |g_A(r')|\ud r' \\
  	&+ \frac 1h\int_r^{r+h}|s(r+h, r')|\cdot|g_A(r')|\ud r'.
\end{aligned}
\end{equation*}
Formula \eqref{eq:fondam} implies that $|s(\wt r, r_0)| \lesssim h$ when $|\wt r-r| \leq h$ and $|r - r_0| \leq h$. Hence, the second term above converges to $0$ as $h\to 0$.
For $0 \leq r_0 \leq r$ and $|\wt r - r| \leq \frac 12 r$ we have the bound $|\frac{\vd^2}{\vd r^2}s(\wt r, r_0)| \lesssim r_0^{-3}$. This implies
\begin{equation*}
  \Bigl|\frac{s(r+h, r')-s(r, r')}{h}-\dd rs(r, r')\Bigr| \leq \frac 12 \sup_{|\wt r-r|\leq h} \big|\frac{\vd^2}{\vd r^2}s(\wt r, r')\big|\cdot |h| \lesssim (r')^{-3}\cdot |h|,
\end{equation*}
so the first term above also converges to $0$ as $h\to 0$.
This shows that $A(r)$ (and similarly $B(r)$) is continuously differentiable and
  \begin{equation}
    \label{eq:dy}
    \begin{aligned}
      A'(r) &= \int_0^r \dd r s(r, r')g_A(r')\ud r', \\
      B'(r) &= \int_0^r \dd r s(r, r')g_B(r')\ud r'.
    \end{aligned}
  \end{equation}
  It is clear from these formulas that $\lim_{r\to 0^+}A'(r) = \lim_{r\to 0^+}B'(r) = 0$.
  
  It follows from the above considerations that $A$ and $B$, seen as functions on $\bR^5$, are $C^1$, so they are $C^\infty$ by elliptic regularity.

  Now we consider the behaviour of $A(r)$ and $B(r)$ as $r\to +\infty$. From the crucial orthogonality relation $\int_0^{+\infty}\Lambda W(r')g_A(r')\ud r' = 0$
  we deduce that
  \begin{equation*}
    \Bigl|\int_0^r \Lambda W(r')g(r')\ud r'\Bigr| = \Bigl|\int_r^{+\infty} \Lambda W(r')g(r')\ud r'\Bigr| \lesssim r^{-1}.
  \end{equation*}
  From this and the asymptotics of $\Gamma$ and $g_A$ it follows that $|A(r)| \lesssim r^{-1}$ and similarly $|B(r)| \lesssim r^{-1}$.
  Using the asymptotics of $\Gamma'$ we obtain also $|A'(r)| \lesssim r^{-2}$ and $|B'(r)| \lesssim r^{-2}$.
  The fact that $|A''(r)| \lesssim r^{-3}$ and $|B''(r)| \lesssim r^{-3}$ follows from the differential equation.
\end{proof}

We define $A$ and $B$ as the solutions of \eqref{eq:AB} satisfying the orthogonality condition
\begin{equation}
  \label{eq:AB-orth}
  \int_{\bR^5} \cZ\cdot A\ud x = \int_{\bR^5} \cZ\cdot B\ud x = 0.
\end{equation}

\subsection{Determination of blow-up speeds}

Let $u^*(t, x)$ be the solution of \eqref{eq:nlw} for initial data $(u^*(0), \partial_t u^*(0)) = (u_0, u_1)$.
At a formal level, while computing the interaction of $u^*$ with the soliton,
we will treat $u^*$ as a function constant in space and $C^2$ in time, $u^*(t, x) \simeq v^*(t)$.
(In the non-degenerate case we will take $v^*(t) = u^*(t, 0)$ and in the degenerate case $v^*(t) = q t^\beta$,
where $q$ and $\beta$ are appropriate constants.)
We will construct a solution which blows up at $t = 0$ and is defined for small positive $t$.
This means that in our situation the caracteristic length $\lambda$ will increase in time.
The usual method of performing a formal analysis of blow-up solutions in the case of the wave equation
consists in defining $b := \lambda_t$ and searching a solution in the form of a power series in~$b$.
Following this scheme, we write
\begin{equation}
  \label{eq:self-nlw}
  \left\{
  \begin{aligned}
    u &= W_\lambda + u^*(t) + b^2T_\lambda + \text{lot} \\
    \partial_t u &= -b(\Lambda W)_{\uln\lambda} + \partial_t u^* + \text{lot}.
  \end{aligned}
  \right.
\end{equation}
Here, the profile $T$ is undetermined, and we search a convenient blow-up speed.
Neglecting irrelevant terms and replacing $\lambda_t := \dd t\lambda(t)$ by $b$, we compute
\begin{equation*}
  \partial_{tt}u = -b_t(\Lambda W)_{\uln\lambda} + \frac{b^2}{\lambda}(\Lambda_0\Lambda W)_{\uln\lambda} + \partial_{tt}u^* + \text{lot}.
\end{equation*}
On the other hand,
\begin{equation*}
  \begin{aligned}
  \Delta u + f(u) &= -\frac{1}{\lambda}b^2(LT)_{\uln\lambda} + f'(W_\lambda)v^* + \Delta u^* + f(u^*) + \text{lot}.
\end{aligned}
\end{equation*}
We discover that, formally at least, we should have
\begin{equation}
  \label{eq:profilT}
  LT = -\Lambda_0\Lambda W + \frac{\lambda}{b^2}[b_t\Lambda W + v^*(t)\sqrt\lambda f'(W)].
\end{equation}
Proposition 
\ref{prop:fredholm} shows that if
\begin{equation}
  \label{eq:b_t}
  b_t = \kappa v^*(t)\lambda^{1/2},
\end{equation}
then equation \eqref{eq:profilT} has a decaying regular solution $T = B + \frac{v^*(t)\lambda^{3/2}}{b^2}A$.
We call equation \eqref{eq:b_t} together with the equation $\lambda_t = b$ \emph{formal parameter equations}.
In the non-degenerate case $v^*(t) = u^*(t, 0)$ is close to $u^*(0, 0)$, so we expect that there exists a solution of
the formal parameter equations which is close to
\begin{equation}\label{eq:mod-app}
  (\lambda(t), b(t)) = \Big(\frac{\kappa^2u^*(0, 0)^2}{144}t^4, \frac{\kappa^2u^*(0, 0)^2}{36}t^3\Big).
\end{equation}
This is indeed the case, as follows from our analysis in Section 
\ref{sec:everr}.

In the degenerate case we have $v^*(t) = q t^\beta$, and the formal parameter equations have a solution
\begin{equation}
  \label{eq:mod-deg}
  (\lambda(t), b(t)) = (t^{1+\nu}, (1+\nu)t^\nu)
\end{equation}
if we choose $q = \frac{\nu(1+\nu)}{\kappa}$ and $\beta = \frac{\nu-3}{2}$.

\section{Approximate solution in the non-degenerate case}
\label{sec:non-deg}
\subsection{Bounds on the profile $(P_0, P_1)$}
The functions $A$ and $B$ from the previous section do not belong to the space $\dot H^1$.
We will place a cut-off at the light cone, that is at distance $t$ from the center. Given modulation parameters $(\lambda(t), b(t))$, we define:
\begin{equation}
  \label{eq:P0}
  P_0(t) := \chi\big(\frac{\cdot}{t}\big)(\lambda(t)^{3/2}v^*(t)A_{\lambda(t)} + b(t)^2 B_{\lambda(t)}).
\end{equation}
Recall that in the non-degenerate case $v^*(t) = u^*(t, 0) \in C^2$ by Proposition 
\ref{prop:cauchy-nondeg} and Schauder estimates.
\begin{remark}
  \label{rem:ustar-cut}
Because of the finite speed of propagation, without loss of generality we can replace $(u_0^*, u_1^*)$ by $\big(\chi\big(\frac{\cdot}{\rho}\big)u_0^*, \chi\big(\frac{\cdot}{\rho}\big)u_1^*\big)$,
where $\rho$ is a strictly positive constant to be chosen later. Thus, without loss of generality we can assume that the support of $(u_0^*, u_1^*)$
is contained in a small ball and that $\|(u_0^*, u_1^*)\|_{X^1\times H^1}$ is small.
\end{remark}
\begin{remark}
The fact that the profile $(P_0, P_1)$ is cut at $r = t = t^1$ can be considered as a coincidence.
The power of $t$ has been chosen in order to optimize the estimates.
This is the only power for which we can obtain the estimate of the error term
which has asymptotically the same size as the profile $P_0$.
Also, for this choice, $\|P_1\|_{L^2}$ (the forth term of the asymptotic expansion
which will be defined in a moment) is asymptotically the same as $\|P_0\|_{\dot H^1}$.
However, the angle of the cone has no significance for us.
\end{remark}
\begin{remark}
  \label{rem:AB-orth}
Notice that the orthogonality condition which we choose to define $A$ and $B$ has little significance
due to a relatively fast decay of $\Lambda W$. We will use the same orthogonality condition as for the error term, as this choice simplifies slightly the computation.
Observe that the fact that $\cZ$ has compact support implies that if $\lambda(t) \ll t$, then $\int P_0(t)\cZ_\lambda\ud x = 0$ for small $t$.
\end{remark}

In the error estimates which will follow, on the right hand side
we will always replace $\lambda(t)$ by $t^4$ and $b(t)$ by $t^3$,
as this is the regime that we are going to consider later in the bootstrap argument.
In this section, all the constants may depend on $\bs u^*$.
\begin{lemma}
  Assume that $\lambda(t) \sim t^4$ and $b(t) \sim t^3$. Then
  \label{lem:size-P0}
  \begin{equation}
    \label{eq:size-P0}
    \|P_0(t)\|_{\dot H^1} \lesssim t^{9/2}.
  \end{equation}
\end{lemma}
\begin{proof}
  It is sufficient to show that $\|\chi(\frac{\cdot}{t})A_\lambda\|_{\dot H^1}^2 \lesssim t^{-3}$
  (the computation for $B_\lambda$ is the same).
  We have
  \begin{equation*}
    \begin{aligned}
      \|\chi\big(\frac{\cdot}{t}\big)A_\lambda\|_{\dot H^1}^2 &\simeq \int_0^{+\infty}\big(\big(\chi\big(\frac{r}{t}\big)A_\lambda(r)\big)'\big)^2 r^4\ud r = \int_0^{+\infty}\big(\big(\chi\big(\frac{\lambda r}{t}\big)A(r)\big)'\big)^2 r^4\ud r\\
      &\lesssim\int_0^{+\infty}\big(\chi\big(\frac{\lambda r}{t}\big)A'(r)\big)^2 r^4\ud r + \int_0^{+\infty}\big(\frac{\lambda}{t}\chi'\big(\frac{\lambda r}{t}\big)A(r)\big)^2 r^4\ud r \\
      &\lesssim \int_0^{2t/\lambda}r^4\frac{1}{r^4}\ud r + \frac{\lambda^2}{t^2}\int_{t/\lambda}^{2t/\lambda}r^4\frac{1}{r^2}\ud r \lesssim \frac{t}{\lambda} \sim t^{-3}.
    \end{aligned}
  \end{equation*}
  
\end{proof}
\begin{lemma}
  Assume that $\lambda(t) \sim t^4$ and $b(t) \sim t^3$. Then
  \label{lem:error-AB}
  \begin{equation}
    \label{eq:error-AB}
    \|L_\lambda P_0 - \lambda^{3/2}v^*(t)L_\lambda A_\lambda - b^2 L_\lambda B_\lambda\|_{L^2} \lesssim t^{7/2}.
  \end{equation}
\end{lemma}
\begin{proof}
  We will do the computation only for the terms with $A$. The terms with $B$ are asymptotically the same.
  We need to check that
\begin{equation*}
  \big\|\big(1-\chi\big(\frac rt\big)\big)f'(W_\lambda)A_\lambda\big\|_{L^2} + \big\|\Delta\big(\big(1-\chi\big(\frac rt\big)\big)A_\lambda\big)\big\|_{L^2} \lesssim t^{-5/2}
\end{equation*}

  For the first term we have even some margin since
  \begin{equation*}
    \begin{aligned}
    \big\|\big(1-\chi\big(\frac rt\big)\big)f'(W_\lambda)A_\lambda\big\|_{L^2} &= \frac{1}{\lambda}\big\|\big(1-\chi\big(\frac{\lambda r}{t}\big)\big)f'(W) A\big\|_{L^2} \\
    &\lesssim \frac{1}{\lambda}\Bigl(\int_{t/\lambda}^{+\infty}(r^{-4}r^{-1})^2r^4\ud r\Bigr)^{1/2}
    \sim \frac{1}{\lambda}\cdot \big(\frac{\lambda}{t}\big)^{5/2} \sim t^{7/2}.
    \end{aligned}
  \end{equation*}
  
  For the second term, we have a few possibilities. Recall that $\Delta = \partial_{rr} + \frac{4\partial_r}{r}$. Either the laplacian hits directly $A$:
  \begin{equation*}
    \big\|\big(1-\chi\big(\frac rt\big)\big)\Delta(A_\lambda)\big\|_{L^2} = \frac{1}{\lambda}\big\|\big(1-\chi\big(\frac{\lambda r}{t}\big)\big)\Delta A\big\|_{L^2} \lesssim \frac{1}{\lambda}\Bigl(\int_{t/\lambda}^{+\infty}(r^{-3})^2r^4\ud r\Bigr)^{1/2}
    \sim \frac{1}{\lambda} \cdot \sqrt\frac{\lambda}{t} \sim t^{-5/2},
  \end{equation*}
  either one derivative hits $\chi$:
  \begin{equation*}
    \frac{1}{t}\big\|\chi'\big(\frac rt\big)\dd r(A_\lambda)\big\|_{L^2} = \frac{1}{t}\big\|\chi'\big(\frac{\lambda r}{t}\big)A'(r)\big\|_{L^2} \lesssim t^{-1}\Bigl(\int_{t/\lambda}^{2t/\lambda}(r^{-2})^2r^4\ud r\Bigr)^{1/2}
    \sim t^{-1} \cdot \sqrt\frac{t}{\lambda} \sim t^{-5/2},
  \end{equation*}
  (and analogously the term $\frac{1}{t}\big\|\chi'\big(\frac rt\big)\frac{4}{r}(A_\lambda)\big\|_{L^2}$),
  or two derivatives hit $\chi$, and we get
  \begin{equation*}
    \frac{1}{t^2}\big\|\chi''\big(\frac rt\big)A_\lambda\big\|_{L^2} = \frac{\lambda}{t^2}\big\|\chi''\big(\frac{\lambda r}{t}\big)A\big\|_{L^2} \lesssim \frac{\lambda}{t^2}\Bigl(\int_{t/\lambda}^{2t/\lambda}(r^{-1})^2r^4\ud r\Bigr)^{1/2}
    \sim \frac{\lambda}{t^2} \cdot \big(\frac{t}{\lambda}\big)^{3/2} \sim t^{-5/2}.
  \end{equation*}
\end{proof}

We define $P_1(t)$ as a formal time derivative of $P_0(t)$,
which means that we replace $\lambda_t$ by $b$ and $b_t$ by $\kappa v^*(t)\lambda^{1/2}$, see \eqref{eq:b_t},
and we do not differentiate the cut-off function. Explicitely, set
\begin{equation}
  \label{eq:P1}
  \begin{aligned}
  P_1(t) &= \chi\big(\frac{\cdot}{t}\big)\Big(v^*(t)\big(\frac{3}{2}\lambda^{3/2}bA_{\uln\lambda} - \lambda^{3/2}b(\Lambda A)_{\uln\lambda}\big) \\
    &+ \lambda^{5/2}\partial_t v^*(t)A_{\uln\lambda} + 2\kappa v^*(t)\lambda^{3/2}bB_{\uln\lambda} - b^3(\Lambda B)_{\uln\lambda}\Big).
  \end{aligned}
\end{equation}

Notice that in the regime \eqref{eq:mod-app} the coefficient $\lambda^{5/2}$ is smaller than the other coefficients
(all of which are, asymptotically, of the same size). However, we prefer to keep the corresponding term in the definition of $P_1$.

\begin{lemma}
  \label{lem:size-P1}
  Assume that $\lambda(t) \sim t^4$ and $b(t) \sim t^3$. Then
  \begin{equation}
    \label{eq:size-P1}
    \|P_1(t)\|_{L^2} \lesssim t^{9/2}
  \end{equation}  
\end{lemma}
\begin{proof}
  All the terms except for the one mentioned above have the same asymptotics, so we will do the computation only for the first one.
  It is sufficient to show that $\|\chi(\frac{\cdot}{t}) A_\uln\lambda\|_{L^2}^2 \lesssim t^{-9}$. We have
  \begin{equation*}
    \begin{aligned}
      \|\chi(\frac{\cdot}{t})A_\uln\lambda\|_{L^2}^2 &\sim \|\chi(\frac{\lambda r}{t})A(r)\|_{L^2(r^4\vd r)}^2 \\
      &\lesssim \int_0^{2t/\lambda}(r^{-1})^2r^4\ud r \lesssim (\frac{t}{\lambda})^3 \sim t^{-9}.
    \end{aligned}
  \end{equation*}
\end{proof}

Our ansatz $\bs \varphi(t) = ( \varphi_0(t),  \varphi_1(t))$ is defined as follows:
\begin{equation}
  \label{eq:ansatz}
  \left\{
  \begin{aligned}
     \varphi_0(t) &= W_{\lambda(t)} + P_0(t) + u^*(t), \\
     \varphi_1(t) &= -b(t)(\Lambda W)_{\uln{\lambda(t)}} + P_1(t) + \partial_t u^*(t),
  \end{aligned}\right.
\end{equation}
where $P_0$ and $P_1$ are given by \eqref{eq:P0} and \eqref{eq:P1}.

The error term $\bs\varepsilon(t) = (\varepsilon_0(t), \varepsilon_1(t))$ is defined by the formula:
\begin{equation}
  \label{eq:epsilon}
  \left\{
  \begin{aligned}
    u(t) &=  \varphi_0(t) + \varepsilon_0(t), \\
    \partial_t u(t) &=  \varphi_1(t) + \varepsilon_1(t).
  \end{aligned}
  \right.
\end{equation}

We shall impose the orthogonality condition
$$\int\varepsilon_0\cZ_\uln\lambda\ud x = 0.$$
\begin{lemma}
  \label{lem:lambda-b}
  If $\lambda \sim t^4$, $b \sim t^3$ and $t$ is small enough, then
  \begin{equation}
    \label{eq:lambda-b}
      |\lambda_t - b| \leq \|\bs\varepsilon\|_\enorm.
  \end{equation}
\end{lemma}
\begin{proof}
To find the formula for $\lambda_t$, first we write
\begin{multline*}
  -b(\Lambda W)_\uln\lambda + \partial_t u^* + P_1(t) + \varepsilon_1(t) = \partial_t u = -\lambda_t(\Lambda W)_\uln\lambda +\partial_t u^* + \partial_t P_0(t) + \partial_t\varepsilon_0 \Rightarrow \\
  \partial_t \varepsilon_0 = (\lambda_t - b)(\Lambda W)_{\uln\lambda} + (P_1 - \partial_t P_0) + \varepsilon_1.
\end{multline*}
Notice that for small $t$ and $\lambda \sim t^4$ we have
\begin{equation*}
  \int(P_1(t) -\partial_t P_0(t))\cZ_\uln\lambda\ud x = (\lambda_t - b)\big(\lambda^{3/2}v^*(t)\la \Lambda A, \cZ\ra_{L^2} + b^2\la \Lambda B, \cZ\ra_{L^2}\big).
\end{equation*}
This follows from \eqref{eq:AB-orth} and the fact that $\mathrm{supp}(\cZ_{\uln\lambda})$ is contained in the light cone for small $t$.
This gives
\begin{equation*}
  \begin{aligned}
  0 &= \dd t \int\varepsilon_0\cZ_\uln\lambda\ud x = \int \partial_t\varepsilon_0\cZ_\uln\lambda\ud x - \lambda_t \int \varepsilon_0\frac{1}{\lambda}(\Lambda_0\cZ)_\uln\lambda\ud x \\
  &= \int (\lambda_t - b)\la\Lambda W, \cZ\ra + (\lambda_t - b)\big(\lambda^{3/2}v^*(t)(\la \Lambda A, \cZ\ra_{L^2} + b^2\la \Lambda B, \cZ\ra_{L^2}\big) \\
  &+ \la\varepsilon_1, \cZ_\uln\lambda\ra_{L^2} -  \lambda_t \int \varepsilon_0\frac{1}{\lambda}(\Lambda_0\cZ)_\uln\lambda\ud x,
\end{aligned}
\end{equation*}
and we obtain
\begin{equation*}
  (\lambda_t - b)\big(\la \Lambda W, \cZ\ra + \lambda^{3/2}v^*(t)\la \Lambda A, \cZ\ra + b^2\la \Lambda B, \cZ\ra \big) = -\la \varepsilon_1, \cZ_\uln\lambda\ra + \lambda_t\big\la \varepsilon_0, \frac{1}{\lambda}(\Lambda_0\cZ)_\uln\lambda\big\ra.
\end{equation*}
Rearranging the terms we get
\begin{equation}
  \label{eq:lambda}
  \begin{aligned}
  \lambda_t =&\Big(1 - \frac{\la \varepsilon_0, \frac{1}{\lambda}(\Lambda_0\cZ)_\uln\lambda\ra}{\la \Lambda W, \cZ\ra_{L^2} + 
  \lambda^{3/2}v^*(t)\la \Lambda A, \cZ\ra + b^2\la \Lambda B, \cZ\ra}\Big)^{-1}\cdot \\
  &\cdot \Big(b - \frac{\la \varepsilon_1, \cZ_\uln\lambda\ra}{\la \Lambda W, \cZ\ra + \lambda^{3/2}v^*(t)\la \Lambda A, \cZ\ra + b^2\la \Lambda B, \cZ\ra}\Big).
\end{aligned}
\end{equation}
For $t$ small enough, \eqref{eq:lambda-b} follows.
\end{proof}
\begin{remark}
  \label{rem:prec-mod}
To be precise, our rigourous argument goes the other way round -- we use \eqref{eq:lambda} and \eqref{eq:b_t}
to \emph{define} the local evolution of the modulation parameters,
and then by doing exactly the same computation as above, but in the opposite direction,
we find that the orthogonality condition $\la\varepsilon_0, \frac{1}{\lambda}\cZ_\uln\lambda\ra_{L^2} = 0$ is preserved
if it is verified at the initial time (which will be the case).
Notice also that using \eqref{eq:AB-orth} we obtain
\begin{equation}
  \label{eq:orth-sansP}
  \la u - W_\lambda - u^*, \cZ_\uln\lambda \ra = 0.
\end{equation}
Differentiating this condition we find
\begin{equation}
  \label{eq:orth-sansP-d}
  \lambda_t\big(\la \Lambda W, \cZ\ra + \big\la \varepsilon_0, \frac{1}{\lambda}(\Lambda_0 \cZ)_\uln\lambda\big\ra = -\la \partial_t u - \partial_t u^*, \cZ_{\uln\lambda}\ra.
\end{equation}
\end{remark}

We need to estimate the error between the formal and the actual time derivative of $P_0$:
\begin{lemma}
  \label{lem:P0t-P1}
  Assume that $\lambda(t) \sim t^4$ and $b(t) \sim t^3$. Then
\begin{equation}
  \label{eq:P0t-P1}
  \|\partial_t P_0-P_1\|_{\dot H^1} \lesssim \sqrt t(t^3 + \|\bs\varepsilon\|_\enorm).
\end{equation}

\end{lemma}
\begin{proof}
  The error has two parts -- one comes from differentiating in time the cut off function
  and the other one from $|\lambda_t - b|$.
  \begin{equation*}
    \begin{aligned}
      \partial_t P_0 - P_1 &= -\frac{r}{t^2}\chi'\big(\frac{r}{t}\big)(\lambda^{3/2}v^*(t)A_\lambda + b^2B_\lambda) \\
      &+ \chi\big(\frac{r}{t}\big)(\lambda_t - b)\big(v^*(t)\big(\frac{3}{2}\lambda^{3/2}A_\uln\lambda - \lambda^{3/2}(\Lambda A)_\uln\lambda\big) - b^2(\Lambda B)_\uln\lambda\big).
    \end{aligned}
  \end{equation*}
  Using Proposition 
\ref{prop:fredholm}, we can write:
  \begin{equation*}
    \begin{aligned}
      &\big\|\frac{r}{\lambda}\chi'\big(\frac{r}{t}\big)A_\lambda\big\|_{\dot H^1} = \big\|r\chi'\big(\frac{\lambda r}{t}\big)A\big\|_{\dot H^1} \\
	&\lesssim \big\|\chi'\big(\frac{\lambda r}{t}\big)\cdot \frac{1}{r}\big\|_{L^2}
      +\frac{\lambda}{t}\big\|r\chi''\big(\frac{\lambda r}{t}\big)\cdot \frac 1r\big\|_{L^2} + \big\|r\chi'\big(\frac{\lambda r}{t}\big)\cdot \frac{1}{r^2}\big\|_{L^2}\\
      &\lesssim \big\|\chi'\big(\frac{\lambda r}{t}\big)\cdot \frac{1}{t}\big\|_{L^2} + \frac{\lambda}{t}\big\|\chi'\big(\frac{\lambda r}{t}\big)\big\|_{L^2} \sim \big(\frac{t}{\lambda}\big)^{3/2} \sim t^{-9/2}.
    \end{aligned}
  \end{equation*}
  The same computation is valid also for $A$ replaced by $B$. Now we have
  $$\|\frac{r}{t^2}\chi'\big(\frac{r}{t}\big)\lambda^{3/2}A_\lambda\|_{\dot H^1} \lesssim \frac{\lambda}{t^2}\lambda^{3/2}\cdot t^{-9/2} \sim t^2 t^6 t^{-9/2} = t^{7/2},$$
  and the same for the second term.

  The computation for the second line is similar:
  \begin{equation*}
    \begin{aligned}
      &\|\chi\big(\frac{r}{t}\big)A_\lambda\|_{\dot H^1} = \|\chi(\frac{\lambda r}{t})A\|_{\dot H^1} \\
      &\lesssim \|\chi\big(\frac{\lambda r}{t}\big)\cdot \frac{1}{r^2}\|_{L^2} + \frac{\lambda}{t}\|\chi\big(\frac{\lambda r}{t}\big)\cdot \frac{1}{r}\|_{L^2} \sim \sqrt{t/\lambda} \sim t^{-3/2}.
    \end{aligned}
  \end{equation*}
  Multiplying by $\sqrt\lambda(\lambda_t - b)$ and using Lemma 
\ref{lem:lambda-b} we get the desired estimate. The last two terms are exactly the same.
\end{proof}

Finally, the following estimate allows to stop the asymptotic expansion of the solution at $P_1$.
\begin{lemma}
  \label{lem:P1t}
  Assume that $\lambda(t) \sim t^4$ and $b(t) \sim t^3$. Then
  \begin{equation}
    \label{eq:P1t}
    \|\partial_t P_1\|_{L^2} \lesssim \sqrt t(t^3 + \|\bs\varepsilon\|_\enorm).
  \end{equation}  
\end{lemma}
\begin{proof}
Consider first the terms coming from differentiating the cut-off function.
Like in the proof of the previous lemma, we have
\begin{equation*}
  \|\frac{r}{\lambda}\chi'\big(\frac rt\big)A_{\uln\lambda}\|_{L^2} \lesssim\|\chi'\big(\frac{\lambda r}{t}\big)\|_{L^2} \sim \big(\frac{t}{\lambda}\big)^{5/2},
\end{equation*}
which gives
\begin{equation*}
  \|\frac{r}{t^2}\chi'\big(\frac rt \big)v^*(t)\lambda^{3/2}bA_{\uln\lambda}\|_{L^2} \lesssim \frac{\lambda}{t^2}\lambda^{3/2}b\cdot\big(\frac{t}{\lambda}\big)^{5/2} \sim t^{7/2}.
\end{equation*}
The term $\|\frac{r}{t^2}\chi'\big(\frac rt\big)\lambda^{5/2}\partial_t v^*(t)A_{\uln\lambda}\|_{L^2}$ is even smaller.

Consider now the other terms. They are of one of the following six types:
\begin{itemize}
  \item $\chi\big(\frac rt\big)\lambda_t\lambda^{1/2}bT_{\uln\lambda}$,
  \item $\chi\big(\frac rt\big)\lambda_t\frac{b^3}{\lambda}T_{\uln\lambda}$,
  \item $\chi\big(\frac rt\big)b_t\lambda^{3/2}T_{\uln\lambda}$,
  \item $\chi\big(\frac rt\big)b_t b^2T_{\uln\lambda}$,
  \item $\chi\big(\frac rt\big)\lambda_t\lambda^{3/2}\vd_t v^*(t)T_{\uln\lambda}$,
  \item $\chi\big(\frac rt\big)\lambda^{5/2}\vd_{tt}v^*(t)T_{\uln\lambda}$,
\end{itemize}
where $T \in \{A, B, \Lambda A, \Lambda B, \Lambda_0 A, \Lambda_0 B, \Lambda_0\Lambda A, \Lambda_0\Lambda B\}$.
In all the situations $T$ is regular and decays like $r^{-1}$ (see Proposition 
\ref{prop:fredholm}), so we can write
\begin{equation*}
  \|\chi\big(\frac rt\big)T_{\uln\lambda}\|_{L^2} \lesssim \Big(\int_{t/\lambda}^{2t/\lambda}\big(\frac 1r\big)^2 r^4\ud r\Big)^{1/2} \lesssim \big(\frac{t}{\lambda}\big)^{3/2}\sim t^{-9/2}.
\end{equation*}
Using the fact that $\lambda \sim t^4$, $b\sim t^3$, $\lambda_t \lesssim b + \|\bs \varepsilon\|$, $b_t \lesssim \sqrt \lambda$ and that $v^*(t)$ is $C^2$ we obtain
\begin{equation*}
  \lambda_t\lambda^{1/2}b + \lambda_t\frac{b^3}{\lambda} + b_t\lambda^{3/2} + b_tb^2 + \lambda_t\lambda^{3/2}|\vd_t v^*| + \lambda^{5/2}|\vd_{tt}v^*| \lesssim t^5(t^3 + \|\bs\varepsilon\|),
\end{equation*}
which finishes the proof.
\end{proof}
The last lemma shows that $\bs \varphi$ is ``almost constant'' after rescaling.
\begin{lemma}
  \label{lem:phi0t}
  Let $c_1 > 0$. If $T_0$ is sufficiently small, then for $t \in (0, T_0]$ there holds
  \begin{equation}
    \label{eq:phi0t}
    \|\partial_t (\varphi_0)_{1/\lambda}\|_{\dot H^1} \leq \frac{c_1}{t}.
  \end{equation}
\end{lemma}
\begin{proof}
  By the definition of $\varphi_0$ and $P_0$ we get
\begin{equation*}
  (\varphi_0)_{1/\lambda} = W + \chi\big(\frac{\lambda\,\cdot}{t}\big)\big(\lambda^{3/2}A + b^2 B\big) + (u^*)_{1/\lambda}.
\end{equation*}

The terms with $A$ and $B$ are similar, so we only consider the first one. We observe that $|\frac{\lambda_t}{\lambda}| \lesssim \frac 1t$ for small $t$,
with an explicit numerical constant. Now
\begin{equation*}
  \partial_t \big(\chi\big(\frac{\lambda r}{t}\big) \lambda^{3/2}A\big) = \frac 32 \frac{\lambda_t}{\lambda}\chi\big(\frac{\lambda r}{t}\big)\lambda^{3/2}A -\frac{\lambda r}{t^2}\chi'\big(\frac{\lambda r}{t}\big)\lambda^{3/2}A.
\end{equation*}
The size of the first term is acceptable by Lemma 
\ref{lem:size-P0}. For the second one, it is sufficient to notice that $|\frac{\lambda r}{t^2}| \leq \frac 2t$ on the support of $\chi$.
The conclusion follows again from Lemma 
\ref{lem:size-P0}. (Notice that we have a large margin for these two terms.)

Next, we have
\begin{equation*}
  \big\|\partial_t (u^*)_{1/\lambda} \big\|_{\dot H^1} \leq \big\|\partial_t u^*\big\|_{\dot H^1} + \frac{\lambda_t}{\lambda}\big\|\Lambda u^*\big\|_{\dot H^1}.
\end{equation*}
By Proposition 
\ref{prop:energy-est} the first term is bounded for small $t$.
Choosing $\rho$ small enough (see Remark 
\ref{rem:ustar-cut}), we can guarantee that $\|\Lambda u^*(t)\|_{\dot H^1}$ will stay small for small $t$, which is exactly what we need.
\end{proof}

\subsection{Error of the ansatz}
Our next objective is to estimate the error of the approximate solution, defined as
\begin{equation*}
  \bs \psi(t) = \begin{pmatrix}\psi_0(t) \\ \psi_1(t) \end{pmatrix} :=
    \begin{pmatrix}\partial_t \varphi_0(t) \\ \partial_t \varphi_1(t)\end{pmatrix} - \begin{pmatrix}\varphi_1(t) \\ \Delta \varphi_0(t) + f(\varphi_0(t))\end{pmatrix}.
\end{equation*}
In order to do this we first need to extract the principal terms of the nonlinear term, which is based on the following pointwise estimate:
\begin{lemma}
  \label{lem:ponctuel}
    \begin{equation}
    \label{eq:ponctuel}
    |f(k+l+m)-[f(k)+f(m)+f'(k)l+f'(k)m]| \lesssim |f(l)|+f'(l)|k| + f'(m)|k| + f'(m)|l|.
  \end{equation}
\end{lemma}
\begin{proof}
  The inequality is homogeneous, so we can suppose that $k^2+l^2+m^2 = 1$.
  The right hand side vanishes only for $(k, l, m) \in  \{(\pm 1, 0, 0), (0, 0, \pm 1)\}$, so it suffices to prove
  the inequality in a neighborhood of these $4$ points, where it is an easy consequence of the Taylor expansion of $f$.
\end{proof}
\begin{lemma}
  If $\lambda(t) \sim t^4$, $b \sim t^3$ and $t$ is small, then
  \label{lem:nonlin-extr}
  \begin{equation}
    \label{eq:nonlin-extr}
    \|f( \varphi_0(t)) - [f(W_{\lambda(t)}) + f(u^*) + f'(W_{\lambda(t)}) P_0(t) + f'(W_{\lambda(t)})u^*(t)]\|_{L^2} \lesssim t^4.
  \end{equation}
\end{lemma}
\begin{proof}
  We put in the preceding lemma $k = W_{\lambda(t)}$, $l = P_0(t)$, $m = u^*(t)$,
  and we estimate the $L^2$ norm of the $4$ terms on the right hand side of \eqref{eq:ponctuel}.
  When $P_0(t)$ appears, we split it into two parts. We sometimes forget $\chi$,
  as its presence here can only help (there are no derivatives).

  Term ``$|f(l)|$'':
  \begin{equation*}
    \big(\chi\big(\frac{r}{t}\big)\lambda^{3/2}A_\lambda\big)^{7/3} \lesssim \chi\big(\frac{r}{t}\big)\cdot\big(\frac{r}{\lambda}\big)^{-7/3},
  \end{equation*}
  and $r^{-14/3}$ is integrable near $0$, so $\|\big(\big(\chi\big(\frac{r}{t}\big)\lambda^{3/2}A_\lambda\big)^{7/3}\|_{L^2} \lesssim \lambda^{7/3}\|\chi\big(\frac{r}{t}\big)r^{-7/3}\|_{L^2} \ll t^4$.
  In a similar way, $\|\big(\chi\big(\frac{r}{t}\big)b^2 B_\lambda\big)^{7/3}\|_{L^2} \ll t^4$.

  Term ``$f'(l)|k|$'':
  By a change of variables we get
  \begin{equation*}
    \|(\lambda^{3/2}A_\lambda)^{4/3}W_\lambda\|_{L^2} = \lambda\|A^{4/3}W\|_{L^2} \sim t^4
  \end{equation*}
  (exponent of $\lambda$ on the left $= (3/2 - 3/2)\cdot(4/3) - 3/2 = -3/2$, and the $L^2$ scaling is $-5/2$).

  In a similar way,
  \begin{equation*}
    \|(b^2B_\lambda)^{4/3}W_\lambda\|_{L^2} = \lambda^{-1}b^{8/3}\|B^{4/3}W\|_{L^2} \sim t^4.
  \end{equation*}

  Term ``$f'(m)|k|$'':
  We use once again the $L^\infty$ bound of $u^*$ and the fact that $\|W_\lambda\|_{L^2} \sim \lambda$.

  Term ``$f'(m)|l|$'':
  Using \eqref{eq:size-P0} and the fact that $u^*(t)$ is bounded in $L^{20/3}$ for small $t$ (by Proposition 
\ref{prop:energy-est}), we have
  \begin{equation*}
    \|f'(u^*)P_0\|_{L^2} \leq \|f'(u^*)\|_{L^5}\cdot \|P_0\|_{L^{10/3}} \lesssim t^{9/2}.
  \end{equation*}
\end{proof}

We can now estimate $\bs \psi(t)$.
\begin{proposition}
  \label{prop:estim-q}
  Assume that $\lambda(t) \sim t^4$ and $b(t) \sim t^3$. Then
  \begin{align}
    \|\psi_0(t) + (\lambda_t - b)\frac{1}{\lambda}(\Lambda W)_\lambda\|_{\dot H^1} & \lesssim \sqrt t(\|\bs \varepsilon(t)\|_\enorm + t^3),\label{eq:estim-q0} \\
    \|\psi_1(t) - (\lambda_t - b)\frac{b}{\lambda}(\Lambda_0\Lambda W)_\uln\lambda\|_{L^2} & \lesssim \sqrt t(\|\bs \varepsilon(t)\|_\enorm+t^3). \label{eq:estim-q1}
  \end{align}
\end{proposition}
\begin{proof}
  The first inequality is just a reformulation of Lemma 
\ref{lem:P0t-P1}.

  For the second inequality, we divide the error into several parts:
  \begin{equation*}
    \begin{aligned}
      \psi_1 &= \partial_t \varphi_1 - (\Delta  \varphi_0 + f( \varphi_0)) \\
      &= (-b_t(\Lambda W)_\uln\lambda + \frac{b\lambda_t}{\lambda}(\Lambda_0 \Lambda W)_\uln\lambda + \partial_t P_1 + \partial_{tt}u^*) \\
      &- (\Delta W_\lambda + \Delta P_0 + \Delta u^*) \\
      &- (f(W_{\lambda}) + f(u^*) + f'(W_{\lambda}) P_0(t) + f'(W_{\lambda})u^*),
    \end{aligned}
  \end{equation*}
  where we have used Lemma 
\ref{lem:nonlin-extr} in order to raplace $f(\varphi_0)$ by the sum of its principal terms.
  Rearranging the terms and using \eqref{eq:b_t}, we can rewrite the sum above as follows:
  \begin{equation*}
    \begin{aligned}
      \psi_1 &=(\lambda_t-b)\frac{b}{\lambda}(\Lambda_0\Lambda W)_\uln\lambda \\
      &- (\Delta W_\lambda + f(W_\lambda)) + (\partial_{tt}u^* - \Delta u^* - f(u^*)) \\
      &- v^*(t)\sqrt\lambda(-LA + \kappa\Lambda W + f'(W))_\uln\lambda + \frac{b^2}{\lambda}(LB + \Lambda_0\Lambda W)_\uln\lambda \\
      &+ (-\Delta P_0 - f'(W_\lambda)P_0)-v^*(t)\sqrt\lambda(LA)_\uln\lambda - \frac{b^2}{\lambda}(LB)_\uln\lambda \\
      &+ (v^*(t) - u^*(t))\sqrt\lambda(f'(W))_\uln\lambda \\
      &+ \partial_t P_1 + O(t^{7/2}).
    \end{aligned}
  \end{equation*}
  Now we proceed line by line.
  \paragraph{Line 1.}
  This is the correction that we substract in \eqref{eq:estim-q1}.
  \paragraph{Line 2.}
  Both terms equal 0.
  \paragraph{Line 3.}
  Both terms equal 0 by the definition of $A$ and $B$.
  \paragraph{Line 4.}
  This error is due to the presence of the cut-off function in \eqref{eq:P0}, and Lemma 
\ref{lem:error-AB} tells us that it is acceptable.
  \paragraph{Line 5.}
  This error is due to the fact that we replace the interaction with $u^*(t)$
  by the interaction with the constant in space function $v^*(t)$.
  It follows from Proposition 
\ref{prop:cauchy-nondeg} that $|v^*(t) - u^*(t, r)| \lesssim r$ uniformly in time when $r\leq t$ and $t$ is small.
  Hence,
  \begin{equation*}
    \|(v^*(t) - u^*(t, r))f'(W_\lambda)\|_{L^2(r\leq t)} \lesssim \|r\sqrt\lambda(f'(W))_\uln\lambda\|_{L^2} \sim \lambda^{3/2}\sim t^6.
  \end{equation*}
  (We have used the fact that $rf'(W)\in L^2$.)
  In the zone $r \geq t$ first we use the fact that $v^*$ is bounded and
  \begin{equation*}
    \|f'(W_\lambda)\|_{L^2(r\geq t)} = \sqrt \lambda\|f'(W)\|_{L^2(r\geq t\lambda^{-1})} \lesssim \sqrt\lambda (\lambda/t)^{3/2} \sim t^{13/2}.
  \end{equation*}
  As for $u^*$, we know from Proposition 
\ref{prop:energy-est} that it is bounded in $L^{10}$. By H\"older
  $\|u^*\cdot f'(W_\lambda)\|_{L^2(r\geq t)} \leq \|u^*\|_{L^{10}}\cdot \|f'(W_\lambda)\|_{L^{5/2}(r\geq t)}$,
  and a routine computation shows that the last term is bounded by $(\lambda/t)^2 \sim t^6$.
  \paragraph{Line 6.}
  This error is small by Lemma 
\ref{lem:P1t}.
\end{proof}
\section{Approximate solution in the degenerate case}
\label{sec:deg}
\subsection{Bounds on the profile $(P_0, P_1)$}
This section is very similar to the previous one. Formula \eqref{eq:P0} is still valid, but recall that in the present case we take
$v^*(t) = q t^\beta$ where $q = \frac{\nu(1+\nu)}{\kappa}$ and $\beta = \frac{\nu-3}{2}$.
The function $u^*_0$ is defined as follows:
\begin{equation}
  \label{eq:ustar-deg}
  u^*_0(x) := \chi\big(\frac{\cdot}{\rho}\big)\cdot p|x|^\beta, \qquad p = \frac{3q}{(\beta+1)(\beta+3)},\ \rho > 0\text{ small}.
\end{equation}
(by the finite speed of propagation, the cut-off does not affect the behaviour at zero for small times, cf. Remark 
\ref{rem:ustar-cut}). We take $u^*_1 = 0$.

In the error estimates which will follow, on the right hand side
we will always replace $\lambda(t)$ by $t^{1+\nu}$ and $b(t)$ by $t^\nu$,
since this is the regime considered later in the bootstrap argument.
\begin{lemma}
  \label{lem:size-P0-deg}
  Assume that $\lambda(t) \sim t^{1+\nu}$ and $b(t) \sim t^\nu$. Then
  \begin{equation}
    \label{eq:size-P0-deg}
    \|P_0(t)\|_{\dot H^1} \lesssim t^{3\nu/2}.
  \end{equation}
\end{lemma}
\begin{proof}
  Recall that $v^*(t) \sim t^\beta = t^{(\nu-3)/2}$, so $\lambda^{3/2}v^*(t) \sim b^2 \sim t^{2\nu}$.
  Hence, it is sufficient to show that $\|\chi(\frac{\cdot}{t})A_\lambda\|_{\dot H^1}^2+\|\chi(\frac{\cdot}{t})B_\lambda\|_{\dot H^1}^2 \lesssim t^{-\nu}$.
  The computation in the proof of Lemma 
\ref{lem:size-P0} gives
  \begin{equation*}
    \begin{aligned}
      \|\chi\big(\frac{\cdot}{t}\big)A_\lambda\|_{\dot H^1}^2 \lesssim \frac{t}{\lambda} \sim t^{-\nu},
    \end{aligned}
  \end{equation*}
  and similarly for the second term.
\end{proof}
\begin{lemma}
  \label{lem:error-AB-deg}
  Assume that $\lambda(t) \sim t^{1+\nu}$ and $b(t) \sim t^\nu$. Then
  \begin{equation}
    \label{eq:error-AB-deg}
    \|L_\lambda P_0 - \lambda^{3/2}v^*(t)L_\lambda A_\lambda - b^2 L_\lambda B_\lambda\|_{L^2} \lesssim t^{3\nu/2-1}.
  \end{equation}
\end{lemma}
\begin{proof}
  We will do the computation only for the terms with $A$. The terms with $B$ are asymptotically the same.
  We need to check that
\begin{equation*}
  \big\|\big(1-\chi\big(\frac rt\big)\big)f'(W_\lambda)A_\lambda\big\|_{L^2} + \big\|\Delta\big(\big(1-\chi\big(\frac rt\big)\big)A_\lambda\big)\big\|_{L^2} \lesssim t^{-\nu/2-1}
\end{equation*}

The computations in the proof of Lemma 
\ref{lem:error-AB} imply that the first term is bounded by $\frac{1}{\lambda}\big(\frac{\lambda}{t}\big)^{5/2}\sim t^{3\nu/2-1}$,
and the second by
\begin{equation*}
  \frac{1}{\lambda}\cdot\sqrt{\frac{\lambda}{t}} + \frac{1}{t}\cdot\sqrt{\frac{t}{\lambda}} + \frac{\lambda}{t^2}\cdot\big(\frac{t}{\lambda}\big)^{3/2} \sim (t\cdot\lambda)^{-1/2} \sim t^{-\nu/2 -1}.
\end{equation*}
\end{proof}

In the degenerate case the profile $P_1(t)$ is defined by the same formula \eqref{eq:P1}.
\begin{lemma}
  \label{lem:size-P1-deg}
  Assume that $\lambda(t) \sim t^{1+\nu}$ and $b(t) \sim t^\nu$. Then
  \begin{equation}
    \label{eq:size-P1-deg}
    \|P_1(t)\|_{L^2} \lesssim t^{3\nu/2}
  \end{equation}  
\end{lemma}
\begin{proof}
  Notice that $\dd t v^*(t) \sim t^{\beta-1} \sim t^{(\nu-5)/2}$. This implies that
  \begin{equation*}
    v^*(t)\cdot\lambda^{3/2}b \sim \dd tv^*(t)\cdot \lambda^{5/2} \sim b^3 \sim t^{3\nu},
  \end{equation*}
  so all the terms in the definition of $P_1(t)$ have asymptotically the same size and it suffices to show that
  $\|\chi(\frac{\cdot}{t}) A_\uln\lambda\|_{L^2}^2 \lesssim t^{-3\nu}$ (the other terms are similar).
  The computation in the proof of Lemma 
\ref{lem:size-P1} gives
  \begin{equation*}
      \big\|\chi\big(\frac{\cdot}{t}\big)A_\uln\lambda\big\|_{L^2}^2 \lesssim \big(\frac{t}{\lambda}\big)^3 \sim t^{-3\nu}.
  \end{equation*}
\end{proof}
Estimate \eqref{eq:lambda-b} and its proof are valid in the degenerate case.
\begin{lemma}
  \label{lem:P0t-P1-deg}
  Assume that $\lambda(t) \sim t^{1+\nu}$ and $b(t) \sim t^\nu$. Then
\begin{equation}
  \label{eq:P0t-P1-deg}
  \|\partial_t P_0-P_1\|_{\dot H^1} \lesssim t^{\nu/2-1}(t^\nu + \|\bs\varepsilon\|_\enorm).
\end{equation}
\end{lemma}
\begin{proof}
  As in the proof of Lemma 
\ref{lem:P0t-P1}, we write
  \begin{equation*}
    \begin{aligned}
      \partial_t P_0 - P_1 &= -\frac{r}{t^2}\chi'(\frac{r}{t})(\lambda^{3/2}v^*(t)A_\lambda + b^2B_\lambda) \\
      &+ \chi(\frac{r}{t})(\lambda_t - b)\big(v^*(t)(\frac{3}{2}\lambda^{3/2}A_\uln\lambda - \lambda^{3/2}(\Lambda A)_\uln\lambda) - b^2(\Lambda B)_\uln\lambda\big).
    \end{aligned}
  \end{equation*}
  The computation in the proof of Lemma 
\ref{lem:P0t-P1} implies
  \begin{equation*}
      \big\|\frac{r}{\lambda}\chi'\big(\frac{r}{t}\big)A_\lambda\big\|_{\dot H^1} \lesssim \big(\frac{t}{\lambda}\big)^{3/2} \sim t^{-3\nu/2}. 
  \end{equation*}
  Multiplying by $\frac{\lambda}{t^2}\lambda^{3/2}v^*(t) \sim t^{3\nu - 1}$ we obtain the required bound on the first term.
  The second term of the first line is similar.

  The second line is bounded exactly as in the proof of Lemma 
\ref{lem:P0t-P1}.
\end{proof}
\begin{lemma}
  \label{lem:P1t-deg}
  Assume that $\lambda(t) \sim t^{1+\nu}$ and $b(t) \sim t^\nu$. Then
  \begin{equation}
    \label{eq:P1t-deg}
    \|\partial_t P_1\|_{L^2} \lesssim t^{\nu/2-1}(t^\nu + \|\bs\varepsilon\|_\enorm).
  \end{equation}  
\end{lemma}
\begin{proof}
  We indicate only the modifications with respect to the proof of Lemma 
\ref{lem:P1t}.
  The term coming from differentiating the cut-off function is estimated as before by
  \begin{equation*}
    \frac{\lambda}{t^2}v^*(t)\lambda^{3/2}b\cdot\big(\frac{t}{\lambda}\big)^{5/2} \sim t^{3\nu/2-1}.
  \end{equation*}
  For the other terms, we get
\begin{equation*}
  \|\chi\big(\frac rt\big)T_{\uln\lambda}\|_{L^2} \lesssim\big(\frac{t}{\lambda}\big)^{3/2}\sim t^{-3\nu/2}.
\end{equation*}
\end{proof}
\subsection{Error of the ansatz}
This subsection differs from the non-degenerate case, because we work here only with $X^1$ regularity
and some more effort is required in order to estimate the terms involving $u^*$.

\begin{lemma}
  \label{lem:nonlin-extr-deg}
  If $\lambda(t) \sim t^{1+\nu}$, $b \sim t^\nu$, $\nu > 8$ and $t$ is small, then
  \begin{equation}
    \label{eq:nonlin-extr-deg}
    \big\|f( \varphi_0(t)) - \big(f(W_{\lambda(t)}) + f(u^*) + f'(W_{\lambda(t)}) P_0(t) + f'(W_{\lambda(t)})u^*(t)\big)\big\|_{L^2} \ll t^{\frac{7}{6}\nu - \frac{7}{3}}.
  \end{equation}
\end{lemma}
\begin{proof}
  As in the proof of Lemma 
\ref{lem:nonlin-extr} we use Lemma 
\ref{lem:ponctuel} with $k = W_{\lambda(t)}$, $l = P_0(t)$ and $m = u^*(t)$.
  We obtain that the $L^2$ norm of the term ``$|f(l)|$'' is bounded by
  $$\big((v^*\cdot\lambda)^{7/3} + b^{14/3}\big)\big\|\chi\big(\frac rt\big)r^{-7/3}\big\|_{L^2},$$
  which is better than required. For the term ``$f'(l)|k|$'' we obtain the bound $(v^*)^{4/3}\cdot \lambda + b^{4/3}\lambda^{-1} \sim t^{5\nu/3-1}$,
  which is again better than required.
  
  Term ``$f'(m)|k|$'':
  Let $(\ulstar, \partial_t \ulstar)$ be the solution of the free wave equation for the initial data $(\ulstar(0), \partial_t \ulstar(0)) = (u_0^*, u_1^*)$.
  We write
  \begin{equation*}
    \|f'(u^*)\cdot W_\lambda\|_{L^2} \lesssim \|f'(\ulstar)\cdot W_\lambda\|_{L^2(|x|\leq \frac 12 t)} 
    + \|f'(u^* - \ulstar)\cdot W_\lambda\|_{L^2(|x|\leq \frac 12 t)} + \|f'(u^*)\cdot W_\lambda\|_{L^2(|x|\geq \frac 12 t)}
  \end{equation*}
  and we examine separately the three terms on the right hand side.
  It follows from Proposition 
\ref{prop:u-lin} that for $|x| \leq \frac 12 t$ we have the bound $|\ulstar(t, x)| \lesssim t^\beta = t^{\frac{\nu-3}{2}}$,
  which implies $\|f'(\ulstar(t))\|_{L^\infty} \lesssim t^{\frac 23(\nu-3)}$, hence
  \begin{equation*}
    \|f'(\ulstar)\cdot W_\lambda\|_{L^2(|x| \leq \frac 12 t)}\lesssim t^{\frac 23(\nu-3)}\|W_\lambda\|_{L^2} \sim t^{\frac 23(\nu-3)} t^{\nu+1} \ll t^{\frac 76\nu-\frac 73}.
  \end{equation*}
  From Proposition 
\ref{prop:u-nonlin} we infer
  \begin{equation*}
    \|u^*-\ulstar\|_{L^{20/3}(|x|\leq \frac 12 t)} \lesssim t^{\frac 76\nu - \frac 73},
  \end{equation*}
  hence
  \begin{equation*}
    \|f'(u^*-\ulstar)\|_{L^{5}(|x| \leq \frac 12 t)} \lesssim t^{\frac{14}{9}\nu - \frac{28}{9}},
  \end{equation*}
  which leads to
  \begin{equation*}
    \|f'(u^*-\ulstar)\cdot W_\lambda\|_{L^2(|x|\leq \frac 12 t)} \leq \|f'(u^*-\ulstar)\|_{L^{5}(|x|\leq \frac 12 t)}\cdot \|W_\lambda\|_{L^{10/3}(|x| \leq \frac 12 t)}
    \lesssim t^{\frac{14}{9}\nu-\frac{28}{9}},
  \end{equation*}
  which is more than sufficient for $\nu > 8$.
  
  For $|x| \geq \frac 12 t$, we know from Proposition 
\ref{prop:energy-est} that $\|f'(u^*)\|_{L^{5}}$ is bounded for small $t$.
  By a change of variables we obtain
  \begin{equation*}
    \|W_\lambda\|_{L^{10/3}(|x| \geq \frac 12 t)} \lesssim \Big(\int_{t/2\lambda}^{+\infty} (r^{-3})^{10/3}r^4\ud r\Big)^{\frac{3}{10}} \sim \big(\frac{\lambda}{t}\big)^{3/2} \ll t^{\frac 76 \nu - \frac 73}.
  \end{equation*}
    
  Term ``$f'(m)|l|$'':
  Using \eqref{eq:size-P0-deg} we have
  \begin{equation*}
    \|f'(u^*)\cdot P_0\|_{L^2} \leq \|f'(u^*)\|_{L^5}\cdot \|P_0\|_{\dot H^1} \lesssim t^{3\nu/2} \ll t^{\frac 76 \nu - \frac 73}.
  \end{equation*}
\end{proof}

We can now estimate $\bs \psi(t)$.
\begin{proposition}
  \label{prop:estim-q-deg}
  Assume that $\lambda(t) \sim t^{1+\nu}$ and $b(t) \sim t^\nu$. Then
  \begin{align}
    \|\psi_0(t) + (\lambda_t - b)\frac{1}{\lambda}(\Lambda W)_\lambda\|_{\dot H^1} & \lesssim t^{\frac 76 \nu - \frac 73} + t^{\nu/2-1}\|\bs\varepsilon(t)\|_\enorm,\label{eq:estim-q0-deg} \\
    \|\psi_1(t) - (\lambda_t - b)\frac{b}{\lambda}(\Lambda_0\Lambda W)_\uln\lambda\|_{L^2} & \lesssim t^{\frac 76 \nu - \frac 73} + t^{\nu/2-1}\|\bs\varepsilon(t)\|_\enorm. \label{eq:estim-q1-deg}
  \end{align}
\end{proposition}
\begin{proof}
  The first inequality follows from Lemma 
\ref{lem:P0t-P1-deg}.

  For the second inequality, as in the proof of Proposition 
\ref{prop:estim-q}, using Lemma 
\ref{lem:nonlin-extr-deg} and rearranging the terms, we get:
  \begin{equation*}
    \begin{aligned}
      \psi_1 &=(\lambda_t-b)\frac{b}{\lambda}(\Lambda_0\Lambda W)_\uln\lambda \\
      &- (\Delta W_\lambda + f(W_\lambda)) + (\partial_{tt}u^* - \Delta u^* - f(u^*)) \\
      &- v^*(t)\sqrt\lambda(-LA + \kappa\Lambda W + f'(W))_\uln\lambda + \frac{b^2}{\lambda}(LB + \Lambda_0\Lambda W)_\uln\lambda \\
      &+ (-\Delta P_0 - f'(W_\lambda)P_0)-v^*(t)\sqrt\lambda(LA)_\uln\lambda - \frac{b^2}{\lambda}(LB)_\uln\lambda \\
      &+ (v^*(t) - u^*(t))\sqrt\lambda(f'(W))_\uln\lambda \\
      &+ \partial_t P_1 + O(t^{\frac 76 \nu - \frac 73}).
    \end{aligned}
  \end{equation*}
  Lines 1, 2, 3, 4 and 6 are treated exactly as in the proof of Proposition 
\ref{prop:estim-q}, using Lemmas 
\ref{lem:error-AB-deg}
  and 
\ref{lem:P1t-deg} instead of Lemmas 
\ref{lem:error-AB} and 
\ref{lem:P1t}.
  We estimate line 5 as follows:
  \begin{equation*}
    \begin{aligned}
      \|(v^* - u^*)f'(W_\lambda)\|_{L^2} &\lesssim \|(v^* - u\lin^*)f'(W_\lambda)\|_{L^2(|x|\leq \frac 12 t)} \\
      &+ \|(u\lin^* - u^*)f'(W_\lambda)\|_{L^2(|x| \leq \frac 12 t)} \\
      &+ \|v^*\cdot f'(W_\lambda)\|_{L^2(|x| \geq \frac 12 t)} \\
      &+ \|u^*\cdot f'(W_\lambda)\|_{L^2(|x| \geq \frac 12 t)}.
    \end{aligned}
  \end{equation*}
  From Proposition 
\ref{prop:u-lin} it follows in particular that $|v^*(t) - u\lin^*(t, r)| \lesssim r$ when $r\leq \frac 12 t$,
  hence the proof of Proposition 
\ref{prop:estim-q} gives the bound
  \begin{equation*}
    \|(v^* - u\lin^*)\cdot f'(W_\lambda)\|_{L^2(|x|\leq \frac 12 t)} \lesssim \lambda^{3/2} \ll t^{\frac 76 \nu - \frac 73}.
  \end{equation*}
  From Proposition 
\ref{prop:u-nonlin} and the fact that $\|f'(W_\lambda)\|_{L^{5/2}} = \|f'(W)\|_{L^{5/2}}$ we get
  \begin{equation*}
    \|(u^* - u\lin^*)\cdot f'(W_\lambda)\|_{L^2(|x|\leq \frac 12 t)} \lesssim t^{\frac 73\beta + \frac 76} = t^{\frac 76\nu - \frac 73}.
  \end{equation*}
  We have
  \begin{equation*}
    \|f'(W_\lambda)\|_{L^2(|x|\geq \frac 12 t)} \lesssim \big\|\frac{\lambda^2}{|x|^4}\big\|_{L^2(|x|\geq \frac 12 t)} \sim \lambda^2 t^{-3/2}
  \end{equation*}
  and
  \begin{equation*}
    \|f'(W_\lambda)\|_{L^{5/2}(|x| \geq \frac 12 t)} \lesssim \big\|\frac{\lambda^2}{|x|^4}\big\|_{L^{5/2}(|x| \geq \frac 12 t)} \sim \lambda^2 t^{-2}.
  \end{equation*}
  Using boundedness of $v^*$ in $L^\infty$, boundedness of $u^*$ in $L^{10}$ and H\"older inequality we obtain the required bounds, which terminates the proof.
\end{proof}

Lemma 
\ref{lem:phi0t} is still valid in the degenerate case, as well as its proof (we use Lemma 
\ref{lem:size-P0-deg} instead of Lemma 
\ref{lem:size-P0}).
\section{Evolution of the error term}
\label{sec:everr}
The evolution of the error term $\bs\varepsilon$ is governed by the following system of differential equations:
\begin{equation}
  \label{eq:err}
  \partial_t
  \begin{pmatrix}
    \varepsilon_0 \\ \varepsilon_1
  \end{pmatrix}
  = \begin{pmatrix}
    \varepsilon_1 - \psi_0 \\
    \Delta \varepsilon_0 + (f(\varphi_0 + \varepsilon_0) - f(\varphi_0)) - \psi_1.
  \end{pmatrix},
\end{equation}
coupled with the equations \eqref{eq:lambda} and \eqref{eq:b_t} for the modulation parameters $\mathrm{Mod} := (\lambda, b)$.
We denote $(T_-, T_+)$ the maximal interval of existence of $\bs u$.

We introduce the energy functional adapted to our ansatz:
\begin{equation}
  \label{eq:I}
  I(t) := \int\frac 12|\varepsilon_1|^2 + \frac 12|\grad \varepsilon_0|^2-(F(\varphi_0 + \varepsilon_0) - F(\varphi_0) - f(\varphi_0)\varepsilon_0)\ud x.
\end{equation}

Essentially, we will perform a bootstrap argument in order to control this functional just by integrating in time its time derivative.
We need a virial correction term which is defined as follows:
\begin{equation}
  \label{eq:J}
  J(t) := b\int\varepsilon_1\cdot\Big(\frac{1}{\lambda}\cdot\frac 12 (\Delta a)_\lambda + (\grad a)_\lambda\cdot\grad\Big)\varepsilon_0\ud x,
\end{equation}
where $a_\lambda(r) = a(\frac{r}{\lambda})$, $(\grad a)_\lambda(r) = \grad a\big(\frac{r}{\lambda}\big)$, $(\Delta a)_\lambda(r) = \Delta a\big(\frac{r}{\lambda}\big)$ and
\begin{equation}
  \label{eq:a}
  a(r) := \left\{
    \begin{aligned}
      &\frac 12 r^2\qquad & |r|\leq R \\
      &\frac{15}{8}Rr-\frac 52 R^2 + \frac 54 R^3 r^{-1} - \frac 18 R^5 r^{-3}\qquad & |r| \geq R
    \end{aligned}\right.
\end{equation}
($R$ is a big radius to be chosen later, see Proposition 
\ref{prop:dtH}).
\begin{lemma}
  \label{lem:fun-a}
  The function $a(r)$ defined above, viewed as a function on $\bR^5$, has the following properties:
  \begin{itemize}
    \item $a \in C^{3,1}$,
    \item $a$ is strictly convex,
    \item $|a(r)| \lesssim r$, $|a'(r)| \lesssim 1$, $|a''(r)| \lesssim r^{-1}$ when $r \to +\infty$ (the constant depends on $R$),
    \item $-\frac{1}{r^3} \lesssim \Delta^2 a(r) \leq 0$.
  \end{itemize}
\end{lemma}
\begin{proof}
  It is apparent from the formula defining $a$ that $a$ is regular except for $r = R$. A computation shows that $a(r)$, $a'(r)$, $a''(r)$ and $a'''(r)$
  are Lipschitz near $r = R$. For $r \geq R$ we have $a''(r) = \frac 52\big(\frac Rr\big)^3 - \frac 32\big(\frac Rr\big)^5 > 0$,
  which proves strict convexity. For $r > R$ one can compute $\Delta^2 a(r) = -\frac{15}{r^3}\cdot \frac{1}{R^2}$
  (where $\Delta = \partial_{rr} + \frac 4r\partial_r$ is the laplacian in dimension $N = 5$).
\end{proof}

We define the mixed energy-virial functional:
\begin{equation*}
  H(t) = I(t) + J(t).
\end{equation*}

The proof of the following result, which will occupy most of this section, is valid both in the non-degenerate and the degenerate case.
The non-degenerate case is obtained for $\nu = 3$. We denote also:
\begin{equation*}
  \gamma := \begin{cases}
    \frac 72 \qquad &\text{in the non-degenerate case,} \\
    \frac 76\nu - \frac 73 \qquad &\text{in the degenerate case,}
  \end{cases}
\end{equation*}
which is the exponent of $t$ in the error estimates in Proposition 
\ref{prop:estim-q} and Proposition 
\ref{prop:estim-q-deg} respectively.

We will use the notation:
\begin{equation*}
  \|\grad_{a, \lambda}\varepsilon_0\|_{L^2}^2 := \int\sum_{i,j}(\partial_{ij}a)_\lambda \partial_i\varepsilon_0\partial_j\varepsilon_0\ud x.
\end{equation*}
\begin{proposition}
  \label{prop:dtH}
  Let $\nu = 3$ or $\nu > 8$.
  Suppose that $\lambda \sim t^{1+\nu}$, $b \sim t^\nu$ and let $c > 0$. If $R$ is chosen large enough, then there exist strictly positive constants $T_0$ and $C_1$ such that
  for $[T_1, T_2] \subset (0, T_0] \cap (T_-, T_+)$ there holds
  \begin{equation}
    \label{eq:dtH}
    \begin{aligned}
      H(T_2) \leq H(T_1) + \int_{T_1}^{T_2}\bigg(&{-}\frac{b}{\lambda}\Big(\|\grad_{a, \lambda}\varepsilon_0\|_{L^2}^2 - \int\big(f(\varphi_0 + \varepsilon_0) - f(\varphi_0)\big)\varepsilon_0\ud x\Big) \\
    &+ \Big(\frac ct\|\bs \varepsilon\|_{\enorm}^2 + C_1t^\gamma\cdot\|\bs\varepsilon\|_\enorm\Big)\bigg)\ud t.
  \end{aligned}
  \end{equation}
\end{proposition}

The proof of this result is going to be an algebraic computation which is not justified in the space $\enorm$.
However, we do not need any uniform control of the regularity or the decay, so we can use the following density argument.
We can approximate a given $\bs\varepsilon$ in $\enorm$ in such a way that the initial data $(u(T_1), \partial_t u(T_1))$ will be in $X^1\times H^1$
and of compact support. Then locally the evolution will have the same proprieties by Proposition 
\ref{prop:persistence},
and will be close to the original one in $\enorm$ for all $t \in [T_1, T_2]$ by local well-posedness in $\enorm$.
The new $\bs \varepsilon$ has sufficient regularity and decay to justify all the computations.
Since the estimate \eqref{eq:dtH} depends continuously (in $\enorm$) on $\bs \varepsilon$, we are done.

We shall split the proof of Proposition 
\ref{prop:dtH} into several Lemmas. We always work under the hypotheses of Proposition 
\ref{prop:dtH}, that is
$\lambda \sim t^{1+\nu}$, $b \sim t^\nu$ and $\|\bs\varepsilon\|_\enorm \leq t^{\gamma + 1}$.
Notice that $\gamma + 1 > \nu$. In the non-degenerate case $\gamma + 1 = \frac 92 > 3 = \nu$ and in the degenerate case
$\gamma + 1 = \frac 76 \nu - \frac 43 > \nu$ because $\nu > 8$. This means that $\|\bs\varepsilon\|_\enorm \ll b$ and
$\|\bs\varepsilon\|_\enorm \ll \frac{\lambda}{t}$ for small $t$.
In what follows $c$ stands for any small strictly positive constant.

We use the method introduced in \cite{RaSz11},
which consists in differentiating the nonlinear term in self-similar variables.
The resulting error will be corrected by the virial term $J$. Concretely, we have:
\begin{equation}\label{eq:self-similar}
  \begin{aligned}
  &\dd t\int\big(F(\varphi_0 + \varepsilon_0) - F(\varphi_0) - f(\varphi_0)\varepsilon_0\big)\ud x \\
  &= \dd t\int\big(F((\varphi_0)_{1/\lambda} + (\varepsilon_0)_{1/\lambda})-F((\varphi_0)_{1/\lambda}) - f((\varphi_0)_{1/\lambda})(\varepsilon_0)_{1/\lambda}\big)\ud x \\
  &= \int\big(f((\varphi_0)_{1/\lambda} + (\varepsilon_0)_{1/\lambda})-f((\varphi_0)_{1/\lambda})-f'((\varphi_0)_{1/\lambda})(\varepsilon_0)_{1/\lambda}\big)\partial_t\big((\varphi_0)_{1/\lambda}\big)\ud x \\
  &+ \int\big(f((\varphi_0)_{1/\lambda} + (\varepsilon_0)_{1/\lambda}) - f((\varphi_0)_{1/\lambda})\big)\big((\varepsilon_{0t})_{1/\lambda} + \frac{\lambda_t}{\lambda}(\Lambda\varepsilon_0)_{1/\lambda}\big)\ud x.
\end{aligned}
\end{equation}
The first term can be neglected, as shown by Lemma 
\ref{lem:phi0t}.
Scaling back the second term we obtain
\begin{equation}
  \label{eq:dtPot}
  \dd t\int\big(F(\varphi_0 + \varepsilon_0) - F(\varphi_0) - f(\varphi_0)\varepsilon_0\big)\ud x \simeq \int\big(f(\varphi_0 + \varepsilon_0)
    - f(\varphi_0)\big)\big(\varepsilon_{0t} + \frac{\lambda_t}{\lambda}\Lambda\varepsilon_0\big)\ud x.
\end{equation}
Here and later the sign $\simeq$ means that the difference of the two sides has size at most $\frac ct\|\bs\varepsilon\|_\enorm^2 + C_1t^\gamma\cdot\|\bs\varepsilon\|_\enorm$.
Also, when we say that a term is ``negligible'', it always means that its absolute value is bounded by $\frac ct\|\bs\varepsilon\|_\enorm^2 + C_1t^\gamma\cdot\|\bs\varepsilon\|_\enorm$.

Using the equations \eqref{eq:err}, \eqref{eq:dtPot} and integrating by parts, we obtain standard cancellations:
\begin{equation}
  \label{eq:dtI}
  \begin{aligned}
    \dd t I(t) &\simeq \int \varepsilon_1 \varepsilon_{1t}\ud x - \int \big(\Delta\varepsilon_0 + f(\varphi_0 + \varepsilon_0) - f(\varphi_0)\big)\varepsilon_{0t}\ud x \\
    &- \frac{\lambda_t}{\lambda}\int\big(f(\varphi_0 + \varepsilon_0) - f(\varphi_0)\big)\Lambda\varepsilon_0\ud x \\
    &\simeq -\int \varepsilon_1 \psi_1 \ud x + \int \big(\Delta\varepsilon_0 + f(\varphi_0 + \varepsilon_0) - f(\varphi_0)\big)\psi_0 \ud x \\
    &- \frac{\lambda_t}{\lambda}\int\big(f(\varphi_0 + \varepsilon_0) - f(\varphi_0)\big)\Lambda\varepsilon_0\ud x.
  \end{aligned}
\end{equation}

Consider now the virial term $J(t)$.
\begin{lemma}
  \label{lem:dtJ}
  \begin{equation}
    \label{eq:dtJ}
    \begin{aligned}
    \dd t J(t) &\leq \int\varepsilon_1 \psi_1\ud x -\frac{b}{\lambda}\|\grad_{a, \lambda}\varepsilon_0\|_{L^2}^2 +  \frac{\lambda_t}{\lambda}\int\big(f(\varphi_0+ \varepsilon_0)-f(\varphi_0)\big)\Lambda_0 \varepsilon_0\ud x \\
    &+ \frac ct\|\bs\varepsilon\|_\enorm^2 + C_1 t^\gamma\cdot\|\bs\varepsilon\|_\enorm.
  \end{aligned}
  \end{equation}
\end{lemma}
Notice the cancellation of $\int\varepsilon_1\psi_1\ud x$ in \eqref{eq:dtI} and \eqref{eq:dtJ}.
This is important because the bound on $\|\psi_1\|_{L^2}$ given by Proposition 
\ref{prop:estim-q} and Proposition 
\ref{prop:estim-q-deg} is only $\frac 1t\|\bs \varepsilon\|_\enorm$,
which is borderline but not sufficient to close the bootstrap. Moreover, $\Lambda_0 - \Lambda = \tx{Id}$, so $J$ eliminates the unbounded part of the operator $\Lambda$ acting on $\varepsilon_0$.

\begin{proof}[Proof of Lemma 
\ref{lem:dtJ}]
  We compute
  \begin{equation*}
    \begin{aligned}
    \dd tJ(t) &= b_t\int\varepsilon_1\cdot\big(\frac{1}{\lambda}\cdot\frac 12 (\Delta a)_\lambda + (\grad a)_\lambda\cdot\grad\big)\varepsilon_0\ud x \\
    &-\frac{b\lambda_t}{\lambda}\int\varepsilon_1\cdot\big(\frac{1}{\lambda}\cdot\frac 12(\Lambda_{3/2}\Delta a)_\lambda + (\Lambda_{5/2}\grad a)_\lambda\cdot\grad \big)\varepsilon_0 \ud x \\
    &+b\int\varepsilon_{1t}\cdot\big(\frac{1}{\lambda}\cdot\frac 12 (\Delta a)_\lambda + (\grad a)_\lambda\cdot\grad\big)\varepsilon_0\ud x \\
    &+b\int\varepsilon_1\cdot\big(\frac{1}{\lambda}\cdot\frac 12 (\Delta a)_\lambda + (\grad a)_\lambda\cdot\grad\big)\varepsilon_{0t}\ud x.
    \end{aligned}
  \end{equation*}
  Consider the first two lines. From Lemma 
\ref{lem:fun-a} and Hardy inequality it follows that
  \begin{equation}
    \label{eq:oper-1}
    \frac{1}{\lambda}\cdot\frac 12 (\Delta a)_\lambda + (\grad a)_\lambda\cdot\grad
  \end{equation}
  and
  $$\frac{1}{\lambda}\cdot\frac 12(\Lambda_{3/2}\Delta a)_\lambda - (\Lambda_{5/2}\grad a)_\lambda\cdot\grad$$
  are uniformly bounded as operators $\dot H^1 \to L^2$ (the bound depends on $R$).
  Moreover, it is clear that $|b_t| + \big|\frac{b\lambda_t}{\lambda}\big| \ll t^{-1}$. Hence, the first two lines are negligible.

  Using again \eqref{eq:err} we get
  \begin{equation}
  \label{eq:viriel-deriv}
    \begin{aligned}
      &b\int\varepsilon_{1t}\cdot\big(\frac{1}{\lambda}\cdot\frac 12 (\Delta a)_\lambda + (\grad a)_\lambda\cdot\grad\big)\varepsilon_0\ud x \\
      &+b\int\varepsilon_1\cdot\big(\frac{1}{\lambda}\cdot\frac 12 (\Delta a)_\lambda + (\grad a)_\lambda\cdot\grad\big)\varepsilon_{0t}\ud x \\
      &= b\int(\Delta\varepsilon_0 + f(\varphi_0 + \varepsilon_0) - f(\varphi_0))\cdot\big(\frac{1}{\lambda}\cdot\frac 12 (\Delta a)_\lambda + (\grad a)_\lambda\cdot\grad\big)\varepsilon_0\ud x \\
      &+b\int\varepsilon_1\cdot\big(\frac{1}{\lambda}\cdot\frac 12 (\Delta a)_\lambda + (\grad a)_\lambda\cdot\grad\big)\varepsilon_1\ud x \\
      &-b\int\psi_1\cdot\big(\frac{1}{\lambda}\cdot\frac 12 (\Delta a)_\lambda + (\grad a)_\lambda\cdot\grad\big)\varepsilon_0\ud x \\
      &-b\int\varepsilon_1\cdot\big(\frac{1}{\lambda}\cdot\frac 12 (\Delta a)_\lambda + (\grad a)_\lambda\cdot\grad\big)\psi_0\ud x.
    \end{aligned}
  \end{equation}
  Proposition 
\ref{prop:estim-q} and Proposition 
\ref{prop:estim-q-deg} imply that $\|\psi_1\|_{L^2} \lesssim \frac 1t \|\bs\varepsilon\|_\enorm + t^\gamma$.
  Using once again uniform boundedness of the operator \eqref{eq:oper-1}, we obtain that the first term of the last line is negligible.
  Consider now the second term. We will show that
  \begin{equation}
    \label{eq:magique}
    \big|b\int\varepsilon_1\cdot\big(\frac{1}{\lambda}\cdot\frac 12 (\Delta a)_\lambda + (\grad a)_\lambda\cdot\grad\big)\psi_0\ud x + \int\varepsilon_1 \psi_1\ud x\big|
    \leq \frac ct\|\bs\varepsilon\|_\enorm^2 + C_1 t^\gamma\cdot\|\bs\varepsilon\|_\enorm.
  \end{equation}

  It follows from Proposition 
\ref{prop:estim-q} and Proposition 
\ref{prop:estim-q-deg} that in \eqref{eq:magique} $\psi_0$ can be replaced by $-(\lambda_t - b)\frac{1}{\lambda}(\Lambda W)_\lambda$
  and $\psi_1$ by $(\lambda_t - b)\frac{b}{\lambda}(\Lambda_0\Lambda W)_{\uln\lambda}$. Hence, using \eqref{eq:lambda-b}, it suffices to prove that
  $\|\Lambda_0 \Lambda W - \bigl[\frac 12 \Delta a + \grad a \cdot \grad\bigr]\Lambda W\|_{L^2}$ is arbitrarily small when $R$ is large enough.
But this is clear, since $\bigl[\frac 12 \Delta a + \grad a \cdot \grad\bigr]\Lambda W(r) = \Lambda_0\Lambda W(r)$ for $r \leq R$
and $|\bigl[\frac 12 \Delta a + \grad a \cdot \grad\bigr]\Lambda W(r)| \lesssim r^{-3}$ for all $r$, with a constant independent of $R$.

The second line of \eqref{eq:viriel-deriv} is $0$ by integration by parts and we are left with the first line.
The term with $\Delta\varepsilon_0$ is computed via a classical Pohozaev identity:
\begin{equation}
  \label{eq:pohozaev}
  \int\Big(\frac{1}{\lambda}\cdot \frac 12(\Delta a)_\lambda + (\grad a)_\lambda\cdot\grad\Big)\varepsilon_0\Delta\varepsilon_0\ud x =
  -\frac{1}{\lambda} \|\grad_{a, \lambda}\varepsilon_0\|_{L^2}^2 +\frac{1}{4\lambda^3}\int (\Delta^2 a)_\lambda\varepsilon_0^2\ud x.
\end{equation}
By Lemma 
\ref{lem:fun-a}, the last term is finite and $\leq 0$.

The nonlinear part is calculated in the following lemma.
\begin{lemma}
  \label{lem:viriel-nonlin}
  \begin{equation}
    \label{eq:viriel-nonlin}
    \begin{aligned}
      \Bigl|&b\int\Big(\frac{1}{\lambda}\cdot\frac 12(\Delta a)_\lambda + (\grad a)_\lambda\cdot\grad\Big)\varepsilon_0\cdot\big(f(\varphi_0 + \varepsilon_0) - f(\varphi_0)\big)\ud x \\
        &- \frac{\lambda_t}{\lambda}\int\big(f(\varphi_0+\varepsilon_0)-f(\varphi_0)\big)\Lambda_0\varepsilon_0\ud x\Bigr| \leq \frac ct \|\bs\varepsilon\|_\enorm^2. 
    \end{aligned}
  \end{equation}
\end{lemma}
We will admit for a moment that this is true and recapitulate in order to finish the proof of Lemma 
\ref{lem:dtJ}.
Identity \eqref{eq:pohozaev} implies that the term with $\Delta \varepsilon_0$ in the first line of \eqref{eq:viriel-deriv}
is \emph{smaller} than $-\frac{b}{\lambda}\|\grad_{a, \lambda}\varepsilon_0\|_{L^2}^2$. 
Lemma 
\ref{lem:viriel-nonlin} implies that the difference between the other term of the first line of \eqref{eq:viriel-deriv} and
$\frac{\lambda_t}{\lambda}\int\big(f(\varphi_0 + \varepsilon_0) - f(\varphi_0)\big)\Lambda_0\varepsilon_0\ud x$ is negligible.
The second line of \eqref{eq:viriel-deriv} is 0, and the difference between the last line and $\int \varepsilon_1\psi_1\ud x$
is negligible, as follows from the computation above.
This proves \eqref{eq:dtJ}.
\end{proof}

In order to prove Lemma 
\ref{lem:viriel-nonlin}, we need two auxiliary facts:
  \begin{lemma}
    \label{lem:ponctuel2}
    \begin{align*}
      |f(k+l)-f(k)-f'(k)l-\frac 12 f''(k)l^2| &\lesssim |f(l)|, \\
      |F(k+l)-F(k)-f(k)l-\frac 12 f'(k)l^2| &\lesssim |F(l)| + |f''(k)|l^3.
    \end{align*}
  \end{lemma}
  \begin{proof}
    For $|l| \leq \frac 12 |k|$ this follows from the Taylor expansion and for $|l|\geq \frac 12 |k|$ this is obvious by the triangle inequality.
  \end{proof}
  \begin{lemma}
    \label{lem:L0phi}
    There exists a constant $C_2$ independent of $R$ such that for small $t$,
    \begin{align}
          \||x|\cdot|\grad\phi_0|\|_{L^{10/3}} &\leq C_2,  \label{eq:viriel-ansatz-1} \\
          \||\lambda(\grad a)_\lambda|\cdot|\grad\phi_0|\|_{L^{10/3}} &\leq C_2 \label{eq:viriel-ansatz-2}.
    \end{align}
    Moreover,
    \begin{equation}
      \label{eq:viriel-loc-err}
  \|\big(x-\lambda(\grad a)_\lambda\big)\cdot\grad\varphi_0\|_{L^{10/3}} \leq c
    \end{equation}
    if $R$ is large enough and $\rho$ small enough.
  \end{lemma}
  \begin{proof}
    Recall that $\varphi_0(t) = W_{\lambda(t)} + P_0(t) + u^*(t)$, and we can estimate the three terms separately. The third one gives $\||x|\cdot|\grad u^*|\|_{L^{10/3}}$,
    which is bounded by Proposition 
\ref{prop:energy-est} and the fact that $u^*$ has compact support.
    It is easy to check that $\||x|\cdot|\grad (W_\lambda)|\|_{L^{10/3}} = \||x|\cdot|\grad W|\|_{L^{10/3}}$,
    which gives the boundedness of the first term. Finally, we compute
    \begin{equation*}
      \grad\Big(\chi\big(\frac xt\big)\lambda^{3/2}A_\lambda(x)\Big) = \grad\Big(\chi\big(\frac xt\big)A(\frac{x}{\lambda})\Big) = \frac 1t(\grad\chi)\big(\frac xt\big)A\big(\frac{x}{\lambda}\big)
      + \frac{1}{\lambda}\chi\big(\frac xt\big)\grad A\big(\frac{x}{\lambda}\big),
    \end{equation*}
    and it is sufficient to use the inequalities $|A(x/\lambda)| \lesssim \lambda / |x|$ and $|\grad A(x/\lambda)| \lesssim \lambda^2 / |x|^2$.
    The second term of $P_0$ is bounded in the same way. Notice that we obtain in fact that $\||x|\cdot |\grad P_0(t)|\|_{L^{10/3}}$ is small when $t$ is small.
    
    Clearly $|\lambda(\grad a)_\lambda| \lesssim |x|$ uniformly in $R$, so \eqref{eq:viriel-ansatz-2} follows from \eqref{eq:viriel-ansatz-1}.

    The proof of \eqref{eq:viriel-loc-err} is similar. The terms $\||x|\cdot|\grad u^*|\|_{L^{10/3}}$ and $|\lambda(\grad a)_\lambda|\cdot|\grad u^*|\|_{L^{10/3}}$ are small
    when $\rho$ is small. By rescaling we get
    \begin{equation*}
      \|\big(x - \lambda(\grad a)_\lambda\big)\cdot|\grad W_\lambda|\|_{L^{10/3}} = \|(x - \grad a)\cdot |\grad W|\|_{L^{10/3}}.
    \end{equation*}
    By definition $\grad a = x$ for $|x| \leq R$, so
    \begin{equation*}
      \|(x - \grad a)\cdot |\grad W|\|_{L^{10/3}} \lesssim \||x|\cdot|\grad W|\|_{L^{10/3}(|x| \geq R)} \to 0\qquad \text{when }R \to +\infty.
    \end{equation*}
    Smallness of $\||\big(x-\lambda(\grad a)_\lambda\big)|\cdot|\grad P_0(t)|\|_{L^{10/3}}$ for small $t$ follows from smallness of $\||x|\cdot |\grad P_0(t)|\|_{L^{10/3}}$. 
  \end{proof}

\begin{proof}[Proof of Lemma 
\ref{lem:viriel-nonlin}]
  First, as for the linear terms, using integration by parts we transform the integral so that the unbounded operator $\Lambda_0$
  (and its approximation $\frac 12 \Delta a + \grad a\cdot \grad$) no longer acts on $\varepsilon_0$:
  \begin{equation}
    \label{eq:viriel-parties}
  \begin{aligned}
    \int \frac{1}{\lambda}x\cdot\grad\varepsilon_0 f(\varphi_0+\varepsilon_0)\ud x &= \int \frac{1}{\lambda}x\cdot\grad(\varphi_0 + \varepsilon_0)f(\varphi_0+\varepsilon_0)\ud x - \int \frac{1}{\lambda}x\cdot\grad\varphi_0 f(\varphi_0+\varepsilon_0)\ud x \\
    &= -5\int \frac{1}{\lambda}F(\varphi_0+\varepsilon_0)\ud x - \int \frac{1}{\lambda}x\cdot\grad\varphi_0 f(\varphi_0+\varepsilon_0)\ud x 
  \end{aligned}
  \end{equation}
  and analogously
  \begin{equation}
    \label{eq:viriel-a-parties}
    \int (\grad a)_\lambda\cdot\grad\varepsilon_0 f(\varphi_0+\varepsilon_0)\ud x = -\int \frac{1}{\lambda}(\Delta a)_\lambda F(\varphi_0+\varepsilon_0)\ud x - \int (\grad a)_\lambda\cdot\grad\varphi_0 f(\varphi_0+\varepsilon_0)\ud x.
  \end{equation}
  Using Lemma 
\ref{lem:ponctuel2} we see that
  \begin{equation*}
    \int\big|F(\varphi_0 + \varepsilon_0) - \big(F(\varphi_0) + f(\varphi_0)\varepsilon_0 + \frac 12 f'(\varphi_0)\varepsilon_0^2\big)\big|\ud x
    \lesssim \|\bs\varepsilon\|_\enorm^3 \leq f(\|\bs\varepsilon\|_\enorm)
  \end{equation*}
  Similarly, from Lemma 
\ref{lem:ponctuel2} and Lemma 
\ref{lem:L0phi} we get
  \begin{equation*}
    \int\big|x\cdot\grad\varphi_0f(\varphi_0 + \varepsilon_0) - x\cdot\grad\varphi_0\big(f(\varphi_0) + f'(\varphi_0)\varepsilon_0 + \frac 12 f''(\varphi_0)\varepsilon_0^2\big)\big|\ud x
    \lesssim f(\|\bs\varepsilon\|_\enorm).
  \end{equation*}
  Notice that $\frac{\lambda_t}{\lambda}f(\|\bs\varepsilon\|_\enorm) \ll \frac{1}{t}\|\bs\varepsilon\|_\enorm^2$, so the above two inequalities
  together with \eqref{eq:viriel-parties} imply that
  \begin{equation}\label{eq:viriel-replace-1}
    \begin{aligned}
    \frac{\lambda_t}{\lambda}\int x\cdot \grad\varepsilon_0 f(\varphi_0 + \varepsilon_0)\ud x \simeq
    &-5\frac{\lambda_t}{\lambda}\int\big(F(\varphi_0) + f(\varphi_0)\varepsilon_0 + \frac 12f'(\varphi_0)\varepsilon_0^2\big)\ud x \\
    &-\frac{\lambda_t}{\lambda}\int x\cdot\grad\varphi_0\big(f(\varphi_0) + f'(\varphi_0)\varepsilon_0 + \frac 12 f''(\varphi_0)\varepsilon_0^2\big)\ud x.
  \end{aligned}
  \end{equation}
  Integrating by parts we find
  \begin{equation*}
    \int x\cdot\grad \varphi_0 f(\varphi_0)\ud x = \int x\cdot\grad F(\varphi_0)\ud x = -5\int F(\varphi_0)\ud x
  \end{equation*}
  and
  \begin{equation*}
    \int x\cdot\grad \varphi_0 f'(\varphi_0)\varepsilon_0\ud x = \int x\cdot\grad f(\varphi_0)\varepsilon_0\ud x = -5\int f(\varphi_0)\varepsilon_0\ud x - \int x\cdot \grad \varepsilon_0 f(\varphi_0)\ud x.
  \end{equation*}
  Thus, \eqref{eq:viriel-replace-1} simplifies to
  \begin{equation}
    \label{eq:viriel-replace-2}
    \frac{\lambda_t}{\lambda}\int x\cdot \grad\varepsilon_0 \big(f(\varphi_0 + \varepsilon_0) - f(\varphi_0)\big)\ud x \simeq
    \frac{\lambda_t}{\lambda}\int\Big({-}\frac 52 f'(\varphi_0)\varepsilon_0^2 - \frac 12 x\cdot \grad\varphi_0 f''(\varphi_0)\varepsilon_0^2\Big)\ud x.
  \end{equation}

Using a pointwise estimate and H\"older we obtain
  \begin{equation*}
    \frac{\lambda_t}{\lambda}\int \varepsilon_0 \big(f(\varphi_0+\varepsilon_0) - f(\varphi_0)\big)\ud x \simeq \frac{\lambda_t}{\lambda}\int f'(\varphi_0)\varepsilon_0^2\ud x.
  \end{equation*}
  Combining with \eqref{eq:viriel-replace-2} we have
  \begin{equation}
    \label{eq:viriel-replace-3}
\begin{aligned}
    \frac{\lambda_t}{\lambda}\int \Lambda_0\varepsilon_0 \big(f(\varphi_0 + \varepsilon_0) - f(\varphi_0)\big)\ud x &\simeq
    -\frac{\lambda_t}{2\lambda}\int x\cdot \grad\varphi_0 f''(\varphi_0)\varepsilon_0^2\ud x  \\ &\simeq -\frac{b}{2\lambda}\int x\cdot \grad\varphi_0 f''(\varphi_0)\varepsilon_0^2\ud x,
\end{aligned}
  \end{equation}
  where the last almost-equality follows from the fact that $|\lambda_t - b| \lesssim \|\bs\varepsilon\|_\enorm$.

  Analogously, we obtain
  \begin{equation}
    \label{eq:viriel-replace-4}
\begin{aligned}
    &b\int \Big(\frac{1}{\lambda}\cdot \frac 12(\Delta a)_\lambda + (\grad a)_\lambda\cdot \grad\Big)\varepsilon_0\big(f(\varphi_0 + \varepsilon_0) - f(\varphi_0)\big)\ud x  \\
&\simeq
-\frac{b}{2}\int (\grad a)_\lambda\cdot \grad\varphi_0 f''(\varphi_0)\varepsilon_0^2\ud x.
\end{aligned}
\end{equation}

Comparing \eqref{eq:viriel-replace-3} and \eqref{eq:viriel-replace-4}, we see that in order to finish the proof, we need to check that
  \begin{equation*}
    \int\bigl|\big(x-\lambda(\grad a)_\lambda\big)\cdot\grad\varphi_0 f''(\varphi_0)\varepsilon_0^2\bigr|\ud x \leq c\|\bs\varepsilon\|_\enorm^2
  \end{equation*}
  when $R$ is sufficiently large.
  Using Sobolev and H\"older inequalities this boils down to
  \begin{equation*}
  \|(x-\lambda(\grad a)_\lambda)\cdot\grad\varphi_0 f''(\varphi_0)\|_{L^{5/2}} \leq c,
\end{equation*}
and this follows from \eqref{eq:viriel-loc-err} and boundedness of $f''(\varphi_0)$ in $L^{10}$.
\end{proof}

\begin{proof}[Proof of Proposition 
\ref{prop:dtH}]
  From \eqref{eq:dtI}, \eqref{eq:dtJ} and the fact that $\Lambda_0 - \Lambda = \mathrm{Id}$, we have
\begin{equation*}
  \begin{aligned}
  \dd t H &= \dd t I + \dd t J \leq \int\big(\Delta\varepsilon_0 + f(\varphi_0+\varepsilon_0)-f(\varphi_0)\big)\psi_0\ud x  \\
  &- \frac{b}{\lambda}\|\grad_{a, \lambda}\varepsilon_0\|_{L^2}^2 + \frac{\lambda_t}{\lambda}\int\big(f(\varphi_0 + \varepsilon_0) - f(\varphi_0)\big)\varepsilon_0\ud x+ \frac ct \|\bs\varepsilon\|_\enorm^2 + C_1 t^\gamma\|\bs\varepsilon\|_\enorm.
  \end{aligned}
\end{equation*}
Notice that
\begin{equation*}
  \|f(\varphi_0 + \varepsilon_0) - f(\varphi_0)\|_{\dot H^{-1}} \lesssim \|\varepsilon_0\|_{\dot H^1}.
\end{equation*}
This follows from the inequality $|f(k+l) - f(k)| \lesssim |l| + |f(l)|$ and the fact that $\varphi_0$ is bounded in $\dot H^1$.
If we recall that $\frac{|\lambda_t - b|}{\lambda} \lesssim \frac{\|\bs\varepsilon\|_\enorm}{\lambda} \ll \frac{1}{t}$,
we see that in the second line we can replace $\lambda_t$ by $b$, hence to finish the proof we only have to prove that
\begin{equation*}
  \int\big(\Delta\varepsilon_0 + f(\varphi_0 + \varepsilon_0) - f(\varphi_0)\big)\psi_0\ud x \leq \frac ct\|\bs\varepsilon\|_\enorm^2 + C_1 t^\gamma\|\bs \varepsilon\|_\enorm.
\end{equation*}
Inequalities \eqref{eq:estim-q0} and \eqref{eq:estim-q0-deg} show that it is sufficient to check that
\begin{equation*}
  \Big|\int\big(\Delta\varepsilon_0 + f(\varphi_0 + \varepsilon_0) - f(\varphi_0)\big)\frac{\lambda_t - b}{\lambda}(\Lambda W)_\lambda\ud x\Big| \leq \frac ct\|\bs\varepsilon\|_\enorm^2 + C_1 t^\gamma\|\bs \varepsilon\|_\enorm,
\end{equation*}
which in turn will follow from \eqref{eq:lambda-b} and
\begin{equation*}
  \Big|\int\big(\Delta\varepsilon_0 + f(\varphi_0 + \varepsilon_0) - f(\varphi_0)\big)(\Lambda W)_\lambda\ud x\Big| \leq \frac{c\lambda}{t}\|\bs\varepsilon\|_\enorm + C_1 \lambda t^{\gamma}.
\end{equation*}
From pointwise bounds (for example the first inequality in Lemma 
\ref{lem:ponctuel2}) one deduces
\begin{equation*}
  \big\|f(\varphi_0 + \varepsilon_0) - f(\varphi_0) - f'(\varphi_0)\varepsilon_0\|_{\dot H^{-1}} \lesssim \|\bs \varepsilon\|_{\dot H^1}^2 \ll \frac{\lambda}{t}\|\bs\varepsilon\|_\enorm,
\end{equation*}
hence it suffices to show that
\begin{equation*}
  \Big|\int\big(\Delta\varepsilon_0 + f'(\varphi_0)\varepsilon_0)(\Lambda W)_\lambda\ud x\Big| =  \Big|\int\varepsilon_0\cdot\big(\Delta + f'(\varphi_0)\big)(\Lambda W)_\lambda\ud x\Big| \leq \frac{c\lambda}{t}\|\bs\varepsilon\|_\enorm + C_1 \lambda t^{\gamma}.
\end{equation*}
Observe that $\big(\Delta + f'(W_\lambda)\big)(\Lambda W)_\lambda = 0$, so we are left with proving that
\begin{equation*}
\Big|\int\varepsilon_0\cdot(f'(\varphi_0) - f'(W_\lambda))(\Lambda W)_\lambda\ud x\Big| \leq \frac{c\lambda}{t}\|\bs\varepsilon\|_\enorm + C_1 \lambda t^{\gamma}.
\end{equation*}
By H\"older inequality it suffices to show that
\begin{equation}
  \label{eq:derniere}
  \big\|\big(f'(\varphi_0) - f'(W_\lambda)\big)(\Lambda W)_\lambda\big\|_{L^{10/7}} \leq \frac{c\lambda}{t}.
\end{equation}
The inequality $|f'(k+l)-f'(k)| \lesssim |f'(l)| + |f''(k)|\cdot|l|$ for $k = W_\lambda$ and $l = u^*(t) + P_0(t)$ reduces \eqref{eq:derniere} to checking that
\begin{align}
  \|f'(u^*)(\Lambda W)_\lambda\|_{L^{10/7}} &\leq \frac{c\lambda}{t}, \label{eq:derniere-1} \\
  \|f'(P_0)(\Lambda W)_\lambda\|_{L^{10/7}} &\leq \frac{c\lambda}{t}, \label{eq:derniere-2} \\
  \|u^*\cdot f'((\Lambda W)_\lambda)\|_{L^{10/7}} &\leq \frac{c\lambda}{t}, \label{eq:derniere-3} \\
  \|P_0\cdot f'((\Lambda W)_\lambda)\|_{L^{10/7}} &\leq \frac{c\lambda}{t}. \label{eq:derniere-4}
\end{align}
Again using H\"older we get $\|P_0\cdot f'((\Lambda W)_\lambda)\|_{L^{10/7}} \leq \|P_0\|_{L^{10/3}} \|f'((\Lambda W)_\lambda)\|_{L^{5/2}} \lesssim \|P_0\|_{\dot H^1}$.
From Lemma 
\ref{lem:size-P0} (or the degenerate version Lemma 
\ref{lem:size-P0-deg}) we have $\|P_0\|_{\dot H^1} \lesssim t^{\gamma + 1} \ll \frac{\lambda}{t}$.
This proves \eqref{eq:derniere-4} and \eqref{eq:derniere-2} is very similar.

From Proposition 
\ref{prop:energy-est} we know that $\|u^*\|_{L^{10}}$ is bounded. Hence
$\|u^*\cdot f'((\Lambda W)_\lambda)\|_{L^{10/7}} \lesssim \|u^*\|_{L^{10}}\cdot \|f'((\Lambda W)_\lambda)\|_{L^{5/3}} \lesssim \lambda$.
This proves \eqref{eq:derniere-3} and \eqref{eq:derniere-1} is similar.
\end{proof}
\section{Construction of a uniformly controlled sequence and conclusion}
\label{sec:shooting}
In this section we will analyse finite dimensional phenomena of our dynamical system --
modulation equations and eigendirections of the linearized operator $L$.
We will also define precisely the bootstrap assumptions and finish the proof of the main theorems.

It is known that the operator $L = -\Delta - f'(W)$ has a unique simple strictly negative eigenvalue $-e_0^2$ (by convention $e_0 > 0$),
with a unique positive eigenfunction $\cY$ such that $\|\cY\|_{L^2} = 1$.
This function $\cY$ is radial, smooth and decays exponentially.
This follows from classical results of spectral theory and theory of elliptic equations, see \cite[Proposition 5.5]{DM08},
where it is also shown that there exists a constant $c_1 > 0$ such that
\begin{equation}
  \label{eq:coer-classic}
  g\in\dot H_{\mathrm{rad}}^1,\quad \la g, \cY\ra = \la \grad g, \grad\Lambda W\ra = 0 \quad\Rightarrow\quad \la g, Lg\ra \geq c_1\|\grad g\|_{L^2}^2.
\end{equation}
We need here a slight modification of this coercivity lemma.
\begin{lemma}
  \label{lem:positive}
  For any $c > 0$ there exist $c_L, C > 0$ such that
  \begin{equation}
    \label{eq:positive}
    \la g, L g\ra \geq c_L\|\grad g\|^2 - C\la g, \cY\ra^2 -c\la g, \cZ\ra^2.
  \end{equation}
\end{lemma}
\begin{proof}
  We first show that
\begin{equation}
  \label{eq:coer-not-classic}
  g\in\dot H_{\mathrm{rad}}^1,\quad \la g, \cY\ra = \la g, \cZ \ra = 0 \quad\Rightarrow\quad \la g, Lg\ra \geq c_2\|\grad g\|_{L^2}^2.
\end{equation}
To prove \eqref{eq:coer-not-classic}, decompose $g= a\Lambda W + h$, $\la h, \Delta \Lambda W\ra = 0$.
Notice that $\la \Lambda W, \cY\ra = 0$, thus $\la h, \cY \ra = 0$ and \eqref{eq:coer-classic} implies
$$
\la g, Lg\ra = \la h + a\Lambda W, L h\ra = \la h, Lh\ra \geq c_1\|\grad h\|_{L^2}^2.
$$
Let $\wt{\Lambda W}$ be the orthogonal projection of $\Delta\Lambda W$ on $\cZ^\perp$ in $\dot H^{-1}$. We have
$$
\begin{aligned}
\|\grad h\|_{L^2}^2 &= \|\grad g - a\grad\Lambda W\|_{L^2}^2 = \|\grad g\|_{L^2}^2 - 2a\la \grad g, \grad \Lambda W\ra + a^2\|\grad \Lambda W\|_{L^2}^2 \\
&= \|\grad g\|_{L^2}^2 + 2a\la g, \wt{\Lambda W}\ra + a^2\|\grad \Lambda W\|_{L^2}^2.
\end{aligned}
$$
The functions $\Delta \Lambda W$ and $\cZ$ are not perpendicular in $\dot H^{-1}$, so $\|\wt{\Lambda W}\|_{\dot H^{-1}} < \|\grad \Lambda W\|_{L^2}$,
and \eqref{eq:coer-not-classic} follows from Cauchy-Schwarz inequality.

In order to prove \eqref{eq:positive}, we decompose 
\begin{equation}
  \label{eq:decomp}
g = a\cY + b\Lambda W + \wt g,\qquad \la \wt g, \cY\ra = \la \wt g, \cZ\ra= 0.
\end{equation}
Projecting \eqref{eq:decomp} on $\cY$ and $\cZ$ we have
  \begin{equation}
    \label{eq:ab-bound}
    \begin{aligned}
      a^2 &\lesssim \la g, \cY \ra^2, \\
      b^2 &\lesssim \la g, \cZ\ra^2 + a^2\la \cZ, \cY \ra^2 \lesssim \la g, \cZ\ra ^2 + \la g, \cY\ra^2.
  \end{aligned}
  \end{equation}
  From \eqref{eq:coer-not-classic} we obtain
  \begin{equation*}
    \la \wt g, L\wt f\ra \geq c_2\|\grad \wt g\|_{L^2}^2,
  \end{equation*}
  thus
  \begin{equation*}
    \la g, L g\ra = \la a\cY + b\Lambda W + \wt g, -e_0^2 a\cY + L\wt g\ra = -e_0^2 a^2 + \la \wt g, L \wt g\ra \geq c_2\|\grad\wt g\|_{L^2}^2 - e_0^2 a^2. 
  \end{equation*}
  From the inequality $(x-y)^2 \geq \frac 12 x^2 - y^2$ we have
  \begin{equation*}
    \|\grad \wt g\|_{L^2}^2 \geq \frac 12 \|\grad g - b\grad \Lambda W\|_{L^2}^2 - a^2\|\grad\cY\|_{L^2}^2.
  \end{equation*}
  From the inequality $(x-y)^2 \geq \frac{c}{1+c}x^2 - cy^2$ we have
  \begin{equation*}
    \|\grad g - b\grad \Lambda W\|_{L^2}^2 \geq \frac{c}{1+c}\|\grad g\|_{L^2}^2 - cb^2\|\grad \Lambda W\|_{L^2}^2.
  \end{equation*}
  If we choose $c$ small enough and put everything together using \eqref{eq:ab-bound}, we obtain \eqref{eq:positive}.
  \end{proof}

From now on we will denote
$$
\alpha(g) := \la g, \cY\ra, \qquad \alpha_\lambda(g) := \Big\la g, \frac{1}{\lambda}\cY_{\uln\lambda}\Big\ra.
$$
We prove a version of the coercivity lemma with a localized gradient term.
\begin{lemma}
  \label{lem:loc-positive}
  Let $c > 0$. If $R$ is large enough, then there exists a constant $C$ such that
  \begin{equation}
    \label{eq:loc-positive}
    \int_{|x|\leq R}|\grad g|^2\ud x - \int_{\bR^5}f'(W)g^2\ud x \geq -c\|\grad g\|_{L^2}^2 - C|\alpha(g)|^2.
  \end{equation}
\end{lemma}
In the proof we assume that $g$ is radial, which is justified because later we use it for $g = \varepsilon_0$.
Notice however that the non-radial case follows by considering the radial rearrangement.
\begin{proof}
  Define the projection $\Psi_R: \dot H^1\to \dot H^1$ by the formula:
  \begin{equation}
    \label{eq:psi_R}
    \Psi_R g(r) = \left\{
      \begin{aligned}
        &g(r) - g(R)\qquad&\text{if }r\leq R, \\
        &0\qquad&\text{if }r\geq R.
      \end{aligned}\right.
  \end{equation}
  By \eqref{eq:positive} applied to $\Psi_R g$ we have
  \begin{equation*}
    \begin{aligned}
    &\big(1+\frac c2\big)\int_{|x|\leq R}|\grad g|^2\ud x = \big(1+\frac c2\big)\int_{\bR^5}|\grad(\Psi_R g)|^2\ud x \\
    &\geq \big(1+\frac c2\big)\int f'(W)|\Psi_R g|^2\ud x - C\la \Psi_R g, \cY\ra^2 - \frac c4 \la \Psi_R g, \cZ\ra^2.
  \end{aligned}
  \end{equation*}

  Recall that, by the Strauss Lemma \cite{Strauss77}, in dimension $N = 5$ for a radial function $g$ we have $|g(R)| \lesssim C_0 R^{-\frac{3}{2}}\|\grad g\|_{L^2}$
with a universal constant $C_0$, so we have a pointwise estimate
  \begin{equation*}
    |g|^2 \leq C_0 \big(1+\frac 2c\big)R^{-3}\|\grad g\|_{L^2}^2+\big(1+\frac c2\big)|\Psi_R g|^2.
  \end{equation*}
  Now we notice that
  \begin{equation*}
    \int_{|x|\leq R}f'(W)\ud x \sim R,
  \end{equation*}
  so for any $\delta > 0$ the first term above gives a small contribution to the quadratic form for $R$ large.
  Similarly,
  \begin{equation*}
    |\la g-\Psi_R g, \cY\ra| \lesssim R^{-\frac{3}{2}}\|\grad g\|_{L^2}+\int_{|x|\geq R} |g|\cY \ud x \leq \bigl(R^{-\frac{3}{2}}+\|\cY\|_{L^{10/7}(|x|\geq R)}\bigr)\|\grad g\|_{L^2},
  \end{equation*}
  which is small when $R$ is large. As $\la \Psi_R g, \cY\ra^2 \leq 2\la g, \cY\ra^2 + 2\la g-\Psi_R g, \cY\ra^2$,
  the proof is finished.
\end{proof}
We are ready to state coercivity properties of the functional $H$ from the previous section.
\begin{proposition}
  \label{prop:coercivity}
  Under the assumptions of Proposition 
\ref{prop:dtH}, there exist $T_0, c_H, \alpha_0, C_2 > 0$ such that for $t \in (0, T_0] \cap (T_-, T_+)$ there holds
  \begin{equation}
    |\alpha_\lambda(\varepsilon_0)| \leq \alpha_0\|\varepsilon_0\|_{\dot H^1} \quad\Rightarrow\quad H(t) \geq c_H\|\bs\varepsilon\|_\enorm^2\label{eq:coer-H}.
  \end{equation}
  If $[T_1, T_2] \subset (0, T_0]\cap (T_-, T_+)$ and $|\alpha_\lambda(\varepsilon_0)| \leq \frac{1}{e_0}t^{\gamma+1}$ for all $t \in [T_1, T_2]$, then
  \begin{equation}
    H(T_2) \leq H(T_1) + \frac{c_H}{10}\int_{T_1}^{T_2}\frac{1}{t}\|\bs\varepsilon\|_\enorm^2 \ud t + C_2 t^{2\gamma+2}.\label{eq:dtH-final}
  \end{equation}
\end{proposition}
The constants $\frac{c_H}{10}$ and $\frac{1}{e_0}$ have no special signification, but this formulation will be convenient later.
\begin{proof}
  Let
  \begin{equation*}
    I\lin(t) := \int \frac 12|\varepsilon_1|^2 + \frac 12 |\grad \varepsilon_0|^2 - \frac 12f'(W_\lambda)\varepsilon_0^2\ud x.
  \end{equation*}
  Recall that $\la \varepsilon_0, \cZ\ra = 0$. Lemma 
\ref{lem:positive} implies (after rescaling) that if we take $\alpha_0$ small enough, then there exists a constant $c > 0$ such that
  $I\lin(t) \geq c \|\bs\varepsilon\|_\enorm^2$.

  We can assume that $\|\bs u^*\|_{X^1 \times H^1}$ is as small as we like, so by pointwise estimates we get $|I(t) - I\lin(t)| \leq \frac 13 c \|\bs \varepsilon_0\|_\enorm^2$.
  Moreover, it is clear from the definition of $J$ that for small $t$ we have $|J(t)| \leq \frac 13 c \|\bs \varepsilon\|_\enorm^2$.
  This proves the result with $c_H = \frac 13 c$.

  In order to prove \eqref{eq:dtH-final}, notice first that, by pointwise estimates and smallness of $\|\varphi_0 - W_\lambda\|_{\dot H^1}$,
  in \eqref{eq:dtH} we can replace $\int \big(f(\varphi_0 + \varepsilon_0) - f(\varphi_0)\big)\varepsilon_0\ud x$ by $\int f'(W_\lambda)\varepsilon_0^2\ud x$.
  Convexity of $a$ (see Lemma 
\ref{lem:fun-a}) implies that
\begin{equation*}
  \|\grad_{a, \lambda}\varepsilon\|_{L^2}^2 \geq \int_{|x| \leq R\lambda}|\grad \varepsilon_0|^2\ud x,
\end{equation*}
so from \eqref{eq:dtH} and Lemma 
\ref{lem:loc-positive} (after rescaling) we obtain
\begin{equation*}
  \begin{aligned}
  H(T_2) &\leq H(T_1) + \int_{T_1}^{T_2}\frac{c_H}{20t}\|\bs\varepsilon\|_\enorm^2 + C_1t^\gamma\cdot\|\bs\varepsilon\|_\enorm + C|\alpha(\varepsilon_0)|^2 \ud t\\
  &\leq H(T_1) + \frac{c_H}{10}\int_{T_1}^{T_2}\frac 1t\|\bs\varepsilon\|_\enorm^2\ud t + C_2 t^{2\gamma+2},
\end{aligned}
\end{equation*}
where $C_2$ is a constant.
\end{proof}

In order to close the bootstrap, it is necessary to control the stable and unstable directions.
More precisely, it is necessary to eliminate the unstable mode.

Define $$\alpha^-_\lambda(\bs\varepsilon) := \int \cY_\uln\lambda\varepsilon_1 - \frac{e_0}{\lambda}\cY_\uln\lambda\varepsilon_0\ud x$$
and $$\alpha^+_\lambda := \int \cY_\uln\lambda\varepsilon_1 + \frac{e_0}{\lambda}\cY_\uln\lambda\varepsilon_0\ud x.$$
Notice that $-\frac{e_0^2}{\lambda^2}$ is the unique strictly negative eigenvalue of $L_\lambda$.

We will define an auxiliary function $l(t)$ which measures the distance of the modulation parameters from the approximate trajectory
\eqref{eq:mod-app} or \eqref{eq:mod-deg}. This function has a slightly different form in the non-degenerate and degenerate cases.
In the non-degenerate case we define
\begin{equation}
  \label{eq:delta}
  l(t) = \frac 12\Bigl(\frac{b}{t^3}+\frac{2\lambda}{t^4} - \frac{\kappa^2u^*(0, 0)^2}{24}\Bigr)^2 + \frac 12\Bigl(\frac{b}{t^3}-\frac{3\lambda}{t^4} - \frac{\kappa^2u^*(0, 0)^2}{144}\Bigr)^2,
\end{equation}
and in the degenerate case
\begin{equation}
  \label{eq:delta-deg}
  l(t) = \frac 12\bigl(\frac{b}{t^\nu}+\wt\nu\frac{\lambda}{t^{\nu+1}} - (\nu+\wt\nu+1)\bigr)^2 + \frac 12\bigl(\frac{b}{t^\nu}-(\wt\nu+1)\frac{\lambda}{t^{\nu+1}} -(\nu-\wt\nu)\bigr)^2,
\end{equation}
where $\wt\nu := -\frac 12 + \frac 12\sqrt{\nu^2 + (\nu+1)^2}$.

We will write $\alpha^+(t)$ and $\alpha^-(t)$ instead of $\alpha^+_\lambda(\bs\varepsilon)$ and $\alpha^-_\lambda(\bs\varepsilon)$.
In the next few propositions we describe the evolution of $\Mod(t) := (\lambda(t), b(t), \alpha^-(t), \alpha^+(t))$ in the  ``modulation cylinder'' defined as:
\begin{equation}
  \label{eq:cylinder}
  \scrC(t) := \big\{(\lambda, b, \alpha^-, \alpha^+):\ l(t) \leq t^{\gamma+1-\nu}\text{ and }-t^{\gamma+1} \leq \alpha^-, \alpha^+ \leq t^{\gamma+1}\big\}.
\end{equation}

In the non-degenerate case we denote
$$\lambda_{\mathrm{app}}(t) := \frac{\kappa^2u^*(0, 0)^2}{144}t^4,\qquad b_{\mathrm{app}}(t) := \frac{\kappa^2u^*(0, 0)^2}{36}t^3,$$
and in the degenerate case
$$\lambda_{\mathrm{app}}(t) := t^{\nu+1},\qquad b_{\mathrm{app}}(t) := (\nu+1)t^\nu.$$
Solving a $2\times 2$ linear system we check easily that
\begin{equation}
  \label{eq:lambda-b-bound}
  l(t) \leq t^{\gamma+1-\nu} \Rightarrow \big|\frac{\lambda}{\lambda_{\mathrm{app}}} - 1\big| \lesssim t^{\frac 12(\gamma+1-\nu)}\ ,\ \big|\frac{b}{b_{\mathrm{app}}} - 1\big| \lesssim t^{\frac 12(\gamma+1-\nu)},
\end{equation}
with constants which depend only on $\nu$.

We have $\alpha_\lambda(\varepsilon_0) = \frac{1}{2e_0}(\alpha^+-\alpha^-)$, so
\begin{equation}
  \label{eq:alpha-bound}
  \Mod(t) \in \scrC(t) \quad \Rightarrow\quad |\alpha_\lambda(\varepsilon_0)| \leq \frac{1}{e_0}t^{\gamma+1}.
\end{equation}

\begin{remark}
The formula for $l$ is found by linearizing the parameter equations near $(\lambda_\tx{app}, b_\tx{app})$ and diagonalizing the resulting system.
\end{remark}

We can finally state a result on uniform in time energy bounds.
\begin{proposition}
  \label{prop:energy-bound}
  Let $C_4 > 0$ be a fixed constant. There exist $C_0 > 0$ and $T_0 > 0$ having the following property.
  Let $0 < T_1 < T_0$. Suppose that
  \begin{align}
      \|\bs\varepsilon(T_1)\|_\enorm &\leq C_4 T_1^{\gamma+1}, \label{eq:boot-size0}\\
      \Mod(T_1) &\in \Int(\scrC(T_1)).\nonumber
  \end{align}
  Then, either there exists a time $t$, $T_1 \leq t \leq T_0$, such that $\Mod(t) \in \partial\scrC(t)$, or the solution exists on $[T_1, T_0]$
  and for all $t \in [T_1, T_0]$ there holds
  \begin{equation}
    \label{eq:boot-size}
      \|\bs\varepsilon\|_\enorm \leq C_0t^{\gamma+1}. 
  \end{equation}
\end{proposition}
\begin{proof}
  Let $T_0$ be the time provided by Proposition 
\ref{prop:coercivity}.
  Let $T_+$ be the maximal time of existence of the solution and let $T_2 := \min(T_0, T_+)$. Suppose that $\Mod(t) \notin \partial\scrC(t)$ for $T_1 \le t \le T_2$.
  By continuity of $\Mod(t)$ this means that $\Mod(t) \in \Int(\scrC(t))$ for $T_1 \le t \le T_2$.
  We will show first that if $C_0$ is large enough, then \eqref{eq:boot-size} holds for $t\in[T_1, T_2]$.
  Argue by contradiction, assuming that there exists $T_3 < T_2$ such that $\|\bs\varepsilon(T_3)\|_\enorm = C_0 T_3^{\gamma+1}$.
  At $t = T_3$ \eqref{eq:alpha-bound} gives $|\alpha_\lambda(\varepsilon_0)| \leq \frac{1}{e_0}t^{\gamma+1}$.
  In particular, if $C_0$ is large, we will have $|\alpha_\lambda(\varepsilon_0)| \leq \alpha_0\|\varepsilon_0\|_{\dot H^1}$,
  so by Proposition 
\ref{prop:coercivity} we obtain
  \begin{equation}
    \label{eq:bootstrap}
      H(T_3) \geq c_HC_0^2 T_3^{2\gamma+2}.
  \end{equation}
  On the other hand, for $t\in[T_1, T_3]$ we have $\|\bs\varepsilon\|_\enorm^2 \leq C_0^2 t^{2\gamma+2}$ and $|\alpha_\lambda(\varepsilon_0)| \leq \frac{1}{e_0}t^{\gamma+1}$,
  so from \eqref{eq:dtH-final} we deduce that 
  $$
  H(T_3) \leq H(T_1) + \frac{c_H C_0^2}{10(2\gamma+2)}T_3^{2\gamma+2} + C_2 T_3^{2\gamma+2}.
  $$
Notice that $H(T_1) \leq \|\bs\varepsilon(T_1)\|_\enorm^2 \leq C_4^2 T_1^{2\gamma+2} \leq C_4^2 T_3^{2\gamma+2}$.
Returning to \eqref{eq:bootstrap} we deduce
$$
c_H C_0^2 \leq C_4^2 + \frac{c_H C_0^2}{10(2\gamma+2)} + C_2,
$$
which is impossible if $C_0$ is large enough.

  Hence, $T_3 = T_2$. To prove that $T_2 = T_0$, notice that by the Cauchy theory in the critical space there exists $\delta > 0$ such that
  \begin{equation*}
    \|(u_0-W, u_1)\|_\enorm \leq \delta \Rightarrow\begin{aligned}&\text{ the solution $\bs u(t)$ with $\bs u(0) = (u_0, u_1)$} \\
      &\text{ exists at least for $t\in(-1, 1)$.}\end{aligned}
  \end{equation*}
  After rescaling we obtain
  \begin{equation}
    \label{eq:existence}
    \|(u_0-W_\lambda, u_1)\|_\enorm \leq \delta \Rightarrow\begin{aligned}&\text{ the solution $\bs u(t)$ with $\bs u(0) = (u_0, u_1)$} \\
      &\text{ exists at least for $t\in(-\lambda, \lambda)$.}\end{aligned}
  \end{equation}
  If $\|\bs u^*\|_\enorm$ is sufficiently small and $T_0$ is chosen sufficiently small, \eqref{eq:size-P0} and \eqref{eq:size-P1}
  show that our solution verifies the sufficient condition in \eqref{eq:existence} for any $t < T_2$ with $\lambda = \lambda(t)$.
  Taking $t$ close to $T_2$ we obtain that the solution cannot blow up at $T_2$, hence $T_2 = T_0$.
\end{proof}
The crucial element of the preceding result is that the constant $C_0$ is independent of $T_1$.
From now, $C_0$ has a fixed value given by Proposition 
\ref{prop:energy-bound},
and the constants which appear later are allowed to depend on $C_0$.
In particular, when we use the notation $\lesssim$ or $O$, the constant may depend on $C_0$.

We examine now the evolution of the eigenvectors $\alpha^-$ and $\alpha^+$.
\begin{lemma}
  \label{lem:eigendirections}
  If $\|\bs \varepsilon\|_\enorm \lesssim t^{\gamma+1}$, $\lambda \sim t^{\nu+1}$ and $b \sim t^\nu$, then
  \begin{equation}
    \label{eq:dtalphap}
    \bigl|\dd t\alpha_\lambda^+ - \frac{e_0}{\lambda}\alpha_\lambda^+\bigr| \lesssim t^\gamma,
  \end{equation}
  \begin{equation}
    \label{eq:dtalpham}
    \bigl|\dd t\alpha^-_\lambda + \frac{e_0}{\lambda}\alpha^-_\lambda\bigr| \lesssim t^\gamma.
  \end{equation}
\end{lemma}
\begin{proof}
  We will do the computation for \eqref{eq:dtalphap}, because the one for \eqref{eq:dtalpham} is exactly the same.
  \begin{equation*}
    \begin{aligned}
      \dd t\alpha_\lambda^+ &= \int\bigl(\cY_\uln\lambda\cdot(-L_\lambda\varepsilon_0) + \frac{e_0}{\lambda}\cY_\uln\lambda\varepsilon_1\bigr)\ud x \\
      &+\int\frac{-\lambda_t}{\lambda}\bigl((\Lambda_0\cY)_\uln\lambda\varepsilon_1 + \frac{e_0}{\lambda}(\Lambda_{-1}\cY)_\uln\lambda\varepsilon_0\bigr)\ud x \\
      &+ \int\cY_\uln\lambda\cdot\big(f(\varphi_0 + \varepsilon_0) - f(\varphi_0) - f'(\varphi_0)\varepsilon_0\big)\ud x \\
      &+ \int\cY_\uln\lambda\cdot\big(f'(\varphi_0) - f'(W_\lambda)\big)\varepsilon_0\ud x \\
      &+ \int\bigl(\cY_\uln\lambda\cdot(-\psi_1) + \frac{e_0}{\lambda}\cY_\uln\lambda\cdot(-\psi_0)\bigr)\ud x.
    \end{aligned}
  \end{equation*}
  The first line is $\frac{e_0}{\lambda}\alpha_\lambda^+$ and it suffices to estimate the remaining ones.
  For the last line we use Proposition 
\ref{prop:estim-q} and $L^2$-orthogonality of $\Lambda W$ and $\cY$.
  Using $\lambda_t \sim t^\nu$, $\lambda \sim t^{\nu+1}$ and $\|\bs\varepsilon\|_\enorm \leq Ct^{\gamma+1}$
  the second line is seen to be bounded by $C t^{\gamma}$.
  The proof of \eqref{eq:derniere} shows that $\|\cY_\uln\lambda\cdot\big(f'(\varphi_0) - f'(W_\lambda)\big)\|_{L^{10/7}} \leq \frac{c}{t}$,
  so using $\|\varepsilon_0\|_{L^{10/3}} \lesssim t^{\gamma+1}$ we obtain the required bound for the fourth line.
  Finally, $\|f(\varphi_0 + \varepsilon_0) - f(\varphi_0) - f'(\varphi_0)\varepsilon_0\|_{\dot H^{-1}} \lesssim C^2 t^{2\gamma+2}$ and $\|\cY_\uln\lambda\|_{\dot H^1} \lesssim \frac{1}{\lambda} \sim t^{-\nu-1}$,
  so by Cauchy-Schwarz the third line is bounded by $C^2 t^{2\gamma-\nu+1} \ll t^\gamma$.
\end{proof}

We know from Proposition 
\ref{prop:energy-bound} that if we start at $t = T_1$ with $\bs\varepsilon$ small enough,
then $\bs\varepsilon$ is controlled in $\enorm$ unless $\Mod$ leaves the cylinder $\scrC$.
It turns out that it can happen only because of $\alpha^+$. The other parameters are trapped in the cylinder for small times:
\begin{lemma}
  \label{lem:modstable}
  Under the assumptions of Propositon 
\ref{prop:energy-bound}, suppose that $\Mod(t)$ leaves $\Int(\scrC(t))$ before $t = T_0$.
  If $T_2 \leq T_0$ is the first time for which $\Mod(T_2) \in \partial\scrC(T_2)$, then $|\alpha^+(T_2)| = T_2^{\gamma+1}$.
  In addition, suppose that at time $T_3$, $T_1 \leq T_3 < T_2$, we have $\alpha^+(T_3) > \frac 12 T_3^{\gamma+1}$. Then $\alpha^+(T_2) = T_2^{\gamma+1}$.
  Analogously, if $\alpha^+(T_3) < -\frac 12 T_3^{\gamma+1}$, then $\alpha^+(T_2) = -T_2^{\gamma+1}$.
\end{lemma}
\begin{proof}
  Suppose, for the sake of contradiction, that for example $l(T_2) = T_2^{\gamma+1-\nu}$. In particular this implies $\dd t l(t_1) \geq 0$, and we will show that it is impossible.
  
  We start with the degenerate case. Using \eqref{eq:lambda-b-bound} and $\sqrt{x} = \frac{1+x}{2} + O(|1-x|^2)$ we obtain
  \begin{equation*}
    \sqrt{\lambda} = \frac 12\big(t^{\frac 12(\nu+1)} + \lambda t^{-\frac 12(\nu+1)}\big) + O(t^{\gamma+\frac 32 - \frac 12\nu}).
  \end{equation*}
  Recall that $b_t = (\nu+1)\nu t^{\frac 12(\nu-3)}\sqrt \lambda$, so we get
  \begin{equation}
    \label{eq:final-bt}
    \frac{b_t}{t^{\nu-1}} = \frac{(\nu+1)\nu}{2}\big(1 + \frac{\lambda}{t^{\nu+1}}\big) + O(t^{\gamma+1-\nu}).
  \end{equation}
  From Lemma 
\ref{lem:lambda-b} and \eqref{eq:boot-size} we have
  \begin{equation}
    \label{eq:final-lambdat}
    \frac{\lambda_t}{t^\nu} = \frac{b}{t^\nu} + O(t^{\gamma + 1-\nu}).
  \end{equation}
  Using \eqref{eq:final-bt} and \eqref{eq:final-lambdat} we can compute $\dd t l(t)$:
  \begin{equation*}
    \begin{aligned}
      &\dd t l(t) = \bigl(\frac{b}{t^\nu} + \wt\nu\frac{\lambda}{t^{\nu+1}} - (\nu+\wt\nu+1)\bigr)\bigl(\frac{b_t}{t^\nu} + \wt\nu\frac{\lambda_t}{t^{\nu+1}} - \nu\frac{b}{t^{\nu+1}} - \wt\nu(\nu+1)\frac{\lambda}{t^{\nu+2}}\bigr) \\
      &+ \bigl(\frac{b}{t^\nu}- (\wt\nu+1) \frac{\lambda}{t^{1+\nu}} - (\nu-\wt\nu)\bigr)\bigl(\frac{b_t}{t^\nu} -(\wt\nu+1) \frac{\lambda_t}{t^{\nu+1}} - \nu\frac{b}{t^{\nu+1}} - (\wt\nu+1)(\nu+1)\frac{\lambda}{t^{\nu+2}}\bigr) \\
      &= \frac{1}{t}\bigl(\frac{b}{t^\nu} + \wt\nu\frac{\lambda}{t^{\nu+1}} - (\nu+\wt\nu+1)\bigr)\bigl(\textstyle{\frac{(\nu+1)\nu}{2}}\big(1 + \frac{\lambda}{t^{\nu+1}}\big)+ \wt\nu\frac{b}{t^{\nu}} - \nu\frac{b}{t^{\nu}} - \wt\nu(\nu+1)\frac{\lambda}{t^{\nu+1}}\bigr) \\
      &+ \bigl(\frac{b}{t^\nu}- (\wt\nu+1) \frac{\lambda}{t^{1+\nu}} - (\nu-\wt\nu)\bigr)\bigl(\textstyle{\frac{(\nu+1)\nu}{2}}\big(1 + \frac{\lambda}{t^{\nu+1}}\big) -(\wt\nu+1) \frac{b}{t^{\nu}} - \nu\frac{b}{t^{\nu}} - (\wt\nu+1)(\nu+1)\frac{\lambda}{t^{\nu+1}}\bigr) \\
    &+ \frac 1t O(\sqrt{l(t)}t^{\gamma+1-\nu}).
  \end{aligned}
  \end{equation*}
  If we use the definition of $\wt \nu$, this simplifies to
  \begin{equation*}
    \begin{aligned}
    \dd t l(t) = &-(\nu-\wt\nu)\bigl(\frac{b}{t^\nu} + \wt\nu\frac{\lambda}{t^{\nu+1}} - (\nu+\wt\nu+1)\bigr)^2 - (\nu + \wt\nu + 1)\bigl(\frac{b}{t^\nu}- (\wt\nu+1) \frac{\lambda}{t^{1+\nu}} - (\nu-\wt\nu)\bigr)^2\\&+ \frac 1t O(\sqrt{l(t)}t^{\gamma+1-\nu})
    = \frac 1t\big(-(\nu-\wt\nu)l(t) + O(t^{\frac 32(\gamma+1-\nu)})\big).
  \end{aligned}
  \end{equation*}
  At time $t = T_2$ by assumption $l(T_2) = T_2^{\gamma+1-\nu}$, so for $T_2$ small enough the formula above implies $\dd t l(T_2) < 0$, which is impossible.

  In the non-degenerate case the computation is similar, but we must take into account that in this case
  \begin{equation*}
    b_t = \kappa u^*(t, 0)\sqrt{\lambda} = \kappa u^*(0, 0)\sqrt \lambda(1 + O(t)),
  \end{equation*}
  which leads to
  \begin{equation*}
    b_t = \kappa u^*(0, 0)\cdot \frac 12\big(\frac{\kappa u^*(0, 0)}{12}t^2 + \frac{12}{\kappa u^*(0, 0)}\lambda t^{-2}\big) + O(t^3).
  \end{equation*}
  Then, the computation is the same as before:
  \begin{equation*}
    \begin{aligned}
    \dd t l(t) &= \bigl(\frac{b}{t^3} + \frac{2\lambda}{t^4} - \frac{\kappa^2 u^*(0, 0)^2}{24}\bigr)\bigl(\frac{b_t}{t^3} + \frac{2\lambda_t}{t^4} - \frac{3b}{t^4} - \frac{8\lambda}{t^5}\bigr) \\
    &+ \bigl(\frac{b}{t^3} + \frac{2\lambda}{t^4} - \frac{\kappa^2 u^*(0, 0)^2}{24}\bigr)\bigl(\frac{b_t}{t^3} + \frac{2\lambda_t}{t^4} - \frac{3b}{t^4} - \frac{8\lambda}{t^5}\bigr) \\
    &\leq \frac{1}{t}\bigl(\frac{b}{t^\nu} + \frac{2\lambda}{t^4} - \frac{\kappa^2 u^*(0, 0)^2}{24}\bigr)\bigl(\frac{\kappa^2 u^*(0, 0)^2}{24}+\frac{6\lambda}{t^4}+\frac{2b}{t^3}-\frac{3b}{t^3} - \frac{8\lambda}{t^4}\bigr) \\
    &+ \frac{1}{t}\bigl(\frac{b}{t^3} - \frac{3\lambda}{t^4} - \frac{\kappa^2 u^*(0, 0)^2}{144}\bigr)\bigl(\frac{\kappa^2 u^*(0, 0)^2}{24}+\frac{6\lambda}{t^4}-\frac{3b}{t^3}-\frac{3b}{t^3} + \frac{12\lambda}{t^4}\bigr)
     + O(\sqrt{l(t)}) \\
     &\leq \frac 1t\big(-2l(t) + O(t^{7/4})\big).
  \end{aligned}
  \end{equation*}
  Since $\gamma + 1 - \nu = \frac 32 < \frac 74$, we are done.

  Now suppose that $|\alpha^-(T_2)| = T_2^{\gamma+1}$. As $\frac{t^{\gamma+1}}{\lambda}\sim t^{\gamma-\nu} \gg t^\gamma$, \eqref{eq:dtalpham} implies that $\dd t\alpha_\lambda^-$ and $\alpha_\lambda^-$
  have opposite signs, which is impossible.

  Again by contradiction, suppose that $\alpha_\lambda^+(T_3) > \frac 12 T_3^{\gamma+1}$ and $\alpha_\lambda^+(T_2) = -T_2^{\gamma+1}$.
  By continuity, there exists the smallest $T_4 > T_3$ such that $\alpha_\lambda^+(T_4) = \frac 12 T_4^{\gamma+1}$.
  Necessarily $\dd t\alpha_\lambda^+(T_4) \leq \frac{\gamma+1}{2} T_4^\gamma$, which is in contradiction with \eqref{eq:dtalphap}.
\end{proof}
\begin{proposition}
  \label{prop:eps-boot}
  There exist strictly positive constants $C_0$ and $T_0$ such that for all $T_1 \in (0, T_0)$ there exists a solution $\bs u$ defined on $[T_1, T_0]$ which for all $t\in [T_1, T_0]$ verifies
  \begin{equation}
    \label{eq:boot-size-final}
    \|(u - W_\lambda - u^*, \partial_t u + \lambda_t (\Lambda W)_\uln\lambda - \partial_t u^*)\|_\enorm \leq C_0 t^{\gamma+1},
  \end{equation}
  \begin{equation}
    \label{eq:lambda-final}
    \big|\frac{\lambda}{\lambda_{\mathrm{app}}}-1\big| \leq C_0 t^{\frac 12(\gamma+1-\nu)}.
  \end{equation}
\end{proposition}
\begin{proof}
  We consider the degenerate case. The proof in the non-degenerate case is similar.

  Let $\lambda = \lambda_{\mathrm{app}}(T_1)$, $b = b_{\mathrm{app}}(T_1)$.
  For $a \in [-\frac 23 T_1^{\gamma+1}, \frac 23 T_1^{\gamma+1}]$, let $\bs\varepsilon_a(T_1) = \frac{a}{2\nu}(\cY_\lambda^+ - \frac{\la\cY, \cZ\ra}{\la \cZ, \cZ\ra}(\cZ_\uln\lambda, 0))$,
  and consider the corresponding evolution. Of course \eqref{eq:boot-size0} is verified for a universal constant $C_4$.
  Let $C_0$ be the constant provided by Proposition 
\ref{prop:energy-bound}.
  We will show that there exists a parameter $a$ for which the solution exists until $t = T_0$ and satisfies \eqref{eq:boot-size}.
  Suppose this is not the case. Let $\mathcal{A}^+ = \{a: \alpha^+(T_2) = T_2^{\gamma+1}\}$ and $\mathcal{A}^- = \{a: \alpha^+(T_2) = -T_2^{\gamma+1}\}$,
  where $T_2$ is the exit time given by Lemma 
\ref{lem:modstable}.
  By the second part of Lemma 
\ref{lem:modstable} we know that $-\frac 23 T_1^{\gamma+1} \in \mathcal{A}^-$, $\frac 23 T_1^{\gamma+1} \in \mathcal{A}^+$, and that $\mathcal{A}^-$ and $\mathcal{A}^+$
  are open sets. Indeed, let $a \in \mathcal{A}^+$. This means in particular that for $T_1 \leq t \leq T_2$ we have $\alpha^+(t) \geq -\frac 12 t^{\gamma+1}$
  and $\alpha^+(T_2) = T_2^{\gamma+1}$. By continuity of the flow, for close enough initial data we will still have $\alpha^+(t) > -t^{\gamma+1}$ for $T_1 \leq t \leq T_2$
  and $\alpha^+(T_2) \geq \frac 12 T_2^{\gamma+1}$. By Lemma 
\ref{lem:modstable} the corresponding solutions escape from the cylinder by positive values of $\alpha^+$.
  Thus $\mathcal{A}^+ \cup \mathcal{A}^-$ would be a partition of $[-\frac 23 T_1^{\gamma+1}, \frac 23 T_1^{\gamma+1}]$ into two disjoint open sets, which is impossible.

  Using \eqref{eq:boot-size}, \eqref{eq:size-P0-deg}, \eqref{eq:size-P1-deg} and \eqref{eq:lambda-b} we obtain \eqref{eq:boot-size-final}.

  Estimate \eqref{eq:lambda-final} follows from \eqref{eq:lambda-b-bound} and the fact that $\Mod(t) \in \scrC(t)$.
\end{proof}

\begin{proof}[Proof of Theorem 
\ref{thm:non-deg}]
  Let $t_n$ be a decreasing sequence such that $t_n > 0$ and $t_n \to 0$. Let $\bs u_n$ be the solution given by Proposition 
\ref{prop:eps-boot} for $T_1 = t_n$
  and let $\lambda_n: [t_n, T_0] \to (0, +\infty)$ be the corresponding modulation parameter.
    The sequence $\bs u_n(T_0)$ is bounded in $\enorm$. After extracting a subsequence, it converges weakly to some function $(u_0, u_1)$.
    Let $\bs u(t)$ be the solution of \eqref{eq:nlw} for the Cauchy data $\bs u(T_0) = (u_0, u_1)$. We will show that $\bs u$ satisfies \eqref{eq:blowup-1}.

    Let $0 < T_1 < T_0$ and $T_1 \leq t \leq T_0$. Using \eqref{eq:boot-size-final}, \eqref{eq:lambda-final} and $|\lambda_t| \lesssim t^3$ we get
    \begin{equation*}
      \|(u_n - W_{\lambda_\mathrm{app}} - u^*, \partial_t u_n - \partial_t u^*)\|_\enorm \leq C_0 t^{\frac 34}.
    \end{equation*}
    This shows that if $T_0$ is sufficiently small, then the sequence $\bs u_n$ satisfies the conditions of Proposition 
\ref{prop:weak-cont}
    on the time interval $[T_1, T_0]$, hence
\begin{equation*}
  \bs u_n(T_1) \wto \bs u(T_1).
\end{equation*}
Weak lower semi-continuity of the norm implies that at time $t = T_1$ we have
    \begin{equation*}
      \|(u - W_{\lambda_\mathrm{app}} - u^*, \partial_t u - \partial_t u^*)\|_\enorm \leq C_0 T_1^{3/4}.
    \end{equation*}
This bound holds for all $T_1$ such that $0 < T_1 < T_0$. In particular, the orthogonality condition:
\begin{equation}
  \label{eq:orth-final}
  \la u - W_\lambda - u^*,\,\cZ_\uln\lambda \ra = 0.
\end{equation}
defines uniquely a continuous function $\lambda(T_1): (0, T_0) \to (0, +\infty)$.
We will prove that $\lambda_n(T_1) \to \lambda(T_1)$.

Using \eqref{eq:orth-sansP} for the solution $u_n$ at time $T_1$ and passing to a limit $n\to \infty$
we obtain that all the accumulation points of $\lambda_n(T_1)$ verify the orthogonality condition \eqref{eq:orth-final}.
Hence $\lambda_n(T_1) \to \lambda(T_1)$.
Passing to a limit in \eqref{eq:orth-sansP-d} we get $\dd t\lambda_n(T_1) \to \dd t\lambda(T_1)$.
Passing to a limit in \eqref{eq:boot-size-final} and \eqref{eq:lambda-final} finishes the proof.
\end{proof}
%

The proof of Theorem \ref{thm:deg} follows the same lines, so we will skip it.

\appendix

\section{Weak continuity of the flow near a fixed path}
\label{sec:weak}
\begin{proposition}
  \label{prop:weak-cont}
  Let $\bs v: [0, 1] \to \enorm$ be a continuous path in the energy space. There exists a constant $\delta > 0$ with the following property.
  Let $\bs u_n$ be a sequence of radial solutions of \eqref{eq:nlw} defined on the interval $[0, 1]$,
  such that
  \begin{equation}
    \label{eq:closeness}
    \sup_{t\in [0,1]}\|\bs u_n - \bs v\|_\enorm \leq \delta.
  \end{equation}
  Suppose that $\bs u_n(0) \wto (u_0, u_1)$ in $\enorm$ and let $\bs u$ be the solution of \eqref{eq:nlw} for the initial data $\bs u(0) = (u_0, u_1)$.
  Then $\bs u$ is defined on $[0, 1]$ and for all $t \in [0, 1]$ we have
  \begin{equation}
    \label{eq:weak-conv}
    \bs u_n(t) \wto \bs u(t)\qquad\text{in }\enorm.
  \end{equation}
\end{proposition}
\begin{remark}
  \label{rem:counterexample}
  Notice that without the assumption \eqref{eq:closeness} the result is false.
  More generally, existence of type II blow-up solutions in some space excludes
  weak continuity of the flow in this space, and existence of type II blowup solutions in our case follows from Theorems 
\ref{thm:non-deg} and 
\ref{thm:deg}.
  One might search weaker conditions than \eqref{eq:closeness}; we have chosen a simple condition which is sufficient for our needs.
\end{remark}
\begin{proof}
  ~\paragraph{Step 1.}
  Suppose that $\bs u$ is not defined on $[0, 1]$ and let $T_+ \leq 1$ be its final time of existence.
  In Step 2. we will prove \eqref{eq:weak-conv} for all $t < T_+$.
  In particular, by the lower weak semi-continuity of the norm, this shows that 
  \begin{equation*}
    \sup_{t\in [0,T_+)}\|\bs u_n - \bs v\|_\enorm \leq \delta.
  \end{equation*}
  By local well-posedness in the energy space and compactness of $\{v(t): t\in[0, 1]\}$, if $\delta > 0$ is small enough,
  there exists $\tau > 0$ such that the solution corresponding to the initial data $\bs u_n(t)$ is defined
  at least on the interval $(-\tau, \tau)$. This means that $\bs u$ cannot blow up at $T_+$, and so it is defined for $t\in [0, 1]$.

  If $\delta$ is chosen small enough, depending on $v(1)$, then by the Cauchy theory the solutions $\bs u_n$ exist on an interval $(1-t', 1+t')$ for some $t' > 0$.
  By eventually choosing $t'$ smaller, we can assume that $\bs u$ also exists on $(1-t', 1+t')$. Hence, by repeating the same procedure we obtain weak convergence also for $t = 1$.
  \paragraph{Step 2.}
  Let $t < T_+$. In order to prove \eqref{eq:weak-conv}, it is sufficient to show that any subsequence of $\bs u_n$ (which we will still denote $\bs u_n$)
  admits a subsequence such that the required convergence takes place.
  By the result of Bahouri-G\'erard a subsequence of $\bs u_n(0)$ admits a profile decomposition such that the first profile is $\bs U\lin\!\!^1(t) = S(t)(u_0, u_1)$
  (corresponding to parameters $t_{1,n} = 0$, $\lambda_{1,n} = 1$). By the triangle inequality $\|\bs u_n - (u_0, u_1)\|_\enorm \leq 2\delta$,
  so all the other profiles are small, in particular they are global and scatter.
  By definition of $T_+$ the assumptions of Proposition {2.8} in \cite{DKM1} (which is a version of \cite[Main Theorem]{BaGe99} for the focusing equation)
  are satisfied for $\theta_n = t$, in particular formula (2.22) from \cite{DKM1} yields:
\begin{equation*}
  \bs u_n(t) = \bs u(t) + \sum_{j=2}^{J}\bs U_n^j(t) + \bs w_n^J(t) + \bs r_n^J(t).
\end{equation*}
Here, $\bs w_n^J(t) = S(t)\bs w_n^J(0) \wto 0$ as $n \to +\infty$ (indeed, $\bs w_n^J(0) \wto 0$ for $J > 1$ by definition of the profiles, and $S(t)$ is a bounded linear operator).
By Lemma 
\ref{lem:small-disperse} below also $U_n^j(t) \wto 0$ when $j > 1$, which finishes the proof.

\end{proof}

\begin{lemma}
  \label{lem:small-disperse}
  Let $\bs U$ be a solution of equation \eqref{eq:nlw} such that $\|\bs U\|_\enorm$ is small.
  Let $t_n, \lambda_n$ be a sequence of parameters such that one of the following holds:
  \begin{enumerate}
    \item{$t_n = 0$ and $\lambda_n \to 0$},
    \item{$t_n = 0$ and $\lambda_n \to +\infty$},
    \item{$t_n/\lambda_n \to +\infty$},
    \item{$t_n/\lambda_n \to -\infty$}.
  \end{enumerate}
  Fix $t \in \bR$ and define
  \begin{equation*}
    \bs U_n(x) = \big(\frac{1}{\lambda_n^{3/2}}U^j(\frac{t-t_n}{\lambda_n},\frac{x}{\lambda_n}), \frac{1}{\lambda_n^{5/2}}\partial_t U^j(\frac{t-t_n}{\lambda_n})\big).
  \end{equation*}
  Then $\bs U_n \wto 0$ in $\enorm$.
\end{lemma}
\begin{proof}
  Again it is sufficient to show this for a subsequence of any subsequence. Thus we can assume that $\frac{t-t_n}{\lambda_n} \to t_0 \in [-\infty,+\infty]$.

  Suppose first that $t_0$ is a finite number. Extracting again a subsequence we can assume that $\lambda_n \to \lambda_0 \in [0, +\infty]$.
  If $\lambda_0$ was a strictly positive finite number, we would obtain that also $t_n$ has a finite limit, which is impossible.
  Thus $\lambda_n \to 0$ or $\lambda_n \to +\infty$, and in both cases we get our conclusion.

  In the case $\tau_n := \frac{t-t_n}{\lambda_n} \to \pm \infty$ we have dispersion, so $\|\bs U_n - (S(\tau_n)\bs V)_{\lambda_n}\|_\enorm \to 0$,
  and it is well known that $(S(\tau_n)\bs V)_{\lambda_n} \wto 0$ when $\tau_n \to \pm\infty$ and $\lambda_n$ is any sequence
  (in the case of space dimension $N = 5$ this follows for example from the strong Huyghens principle).
\end{proof}

\section{Local theory in higher regularity}
\label{sec:cauchy}

In this section we use the energy method to prove two results about preservation of regularity.
\subsection{Energy estimates in $X^1\times H^1$}
Recall that we denote $X^s := \dot H^{s+1} \cap \dot H^1$. We have classical \emph{energy estimates} for the linear wave equation:
\begin{lemma}
  \label{lem:energy-est}
  Let $s \in \bN$. Let $I = [0, T_0]$ be a time interval, $g\in C(I, H^s)$ and $(u_0, u_1) \in X^s\times H^s$. Then the Cauchy problem
  \begin{equation*}
    \begin{cases}
      \partial_{tt}u - \Delta u = g, \\
      (u(0), \partial_t u(0)) = (u_0, u_1)
    \end{cases}
  \end{equation*}
  has a unique solution $(u, \partial_t u) \in C(I, X^s \times H^s)$ and for all $t \in I$ there holds
  \begin{equation}
    \label{eq:energy-lin}
    \|(u, \partial_t u)\|_{X^s\times H^s} \leq \|(u_0, u_1)\|_{X^s\times H^s} + \int_0^t\|g(\tau)\|_{H^s}\ud \tau.
  \end{equation}
\end{lemma}
For a proof of a more general result one can consult for example \cite[Theorem 4.4]{BaChDa11}.
Using finite speed of propagation and Sobolev Extension Theorem on each time slice we get a localised version of the energy estimate:
\begin{equation}
  \label{eq:energy-lin-loc}
  \|(u, \partial_t u)\|_{X^s\times H^s(B(0, \rho))} \lesssim \|(u_0, u_1)\|_{X^s\times H^s(B(0, \rho + t))} + \int_0^t\|g(\tau)\|_{H^s(B(0, \rho + \tau)}\ud\tau
\end{equation}

Now we use the case $s= 1$ to prove energy estimates in $X^1\times H^1$ for \eqref{eq:nlw}.
\begin{proposition}
  \label{prop:energy-est}
  For all $M_0 > 0$ there exists $T_0 = T_0(M_0) > 0$ such that the following is true.
  Let $(u_0, u_1) \in X^1\times H^1$ with $\|(u_0, u_1)\|_{X^1\times H^1} \leq M_0$.
  Then the Cauchy problem:
  \begin{equation*}
    \begin{cases}
      \partial_{tt}u - \Delta u = f(u), \\
      (u(0), \partial_t u(0)) = (u_0, u_1)
    \end{cases}
  \end{equation*}
  has a unique solution $(u, \partial_t u)\in C([0, T_0], X^1\times H^1)$ and this solution verifies
  \begin{equation}
    \label{eq:energy-nonlin-1}
    \sup_{t\in[0, T_0]}\|(u(t), \partial_t u(t))\|_{X^1 \times H^1} \leq 2\|(u_0, u_1)\|_{X^1\times H^1}.
  \end{equation}
  Moreover, let $u\lin$ denote the solution of the free wave equation for the same initial data $(u_0, u_1)$. Then
  \begin{equation}
    \label{eq:energy-nonlin-2}
    \sup_{t\in[0, T_0]}\|(u(t), \partial_t u(t)) - (u\lin(t), \partial_t u\lin L(t))\|_{X^1\times H^1} \lesssim f\big(\|(u_0, u_1)\|_{X^1\times H^1}\big).
  \end{equation}
\end{proposition}
This will follow easily from the following lemma.
\begin{lemma}
  \label{lem:contraction}
  Let $u, v \in X^1$. Then
  \begin{align}
    \|f(u)\|_{H^1} &\leq C f(\|u\|_{X^1}), \label{eq:contr1} \\
    \|f(u) - f(v)\|_{H^1} &\leq C\|u-v\|_{X^1}\cdot\big(f'(\|u\|_{X^1}) + f'(\|v\|_{X^1})\big) \label{eq:contr2}.
  \end{align}
\end{lemma}
\begin{proof}
  We have $\|f(u)\|_{L^2} = \|u\|_{L^{14/3}}^{7/3} \lesssim f(\|u\|_{X^1})$ from the Sobolev imbedding.
  By H\"older inequality,
  \begin{equation*}
    \|\grad f(u)\|_{\dot H^1} = \|\grad u \cdot f'(u)\|_{L^2} \lesssim \|\grad u\|_{L^{10/3}}\cdot\|f'(u)\|_{L^5} \lesssim \|u\|_{X^1}\cdot \|u\|_{L^{20/3}}^{4/3} \lesssim f(\|u\|_{X^1}),
  \end{equation*}
  again by Sobolev imbedding. This proves \eqref{eq:contr1}.

  To prove \eqref{eq:contr2}, we write $|f(u)-f(v)| \lesssim |u-v|(f'(u) + f'(v))$, hence
  \begin{equation*}
    \begin{aligned}
    \|f(u)-f(v)\|_{L^2} &\lesssim \|u-v\|_{L^{14/3}}\cdot\|f'(u) + f'(v)\|_{L^{7/2}} \lesssim \|u-v\|_{L^{14/3}}\cdot\big(\|u\|_{L^{14/3}}^{4/3} + \|v\|_{L^{14/3}}^{4/3}\big) \\
    &\lesssim \|u-v\|_{X^1}\cdot\big(f'(\|u\|_{X^1}) + f'(\|v\|_{X^1})\big).
  \end{aligned}
  \end{equation*}
  Finally,
  \begin{equation*}
    |\grad f(u) - \grad f(v)| \lesssim |\grad u - \grad v|(f'(u) + f'(v)) + |u-v|(|\grad u| + |\grad v|)(|f''(u)| + |f''(v)|),
  \end{equation*}
  and it suffices to notice that
  \begin{equation*}
    \begin{aligned}
    \| |\grad u - \grad v|(f'(u) + f'(v)) \|_{L^2} &\lesssim \|\grad u - \grad v\|_{L^{10/3}}\cdot \|f'(u) + f'(v)\|_{L^5} \\ 
    &\lesssim \|u-v\|_{X^1}\cdot\big(f'(\|u\|_{X^1}) + f'(\|v\|_{X^1})\big)
  \end{aligned}
  \end{equation*}
  and
  \begin{equation*}
    \begin{aligned}
      &\bigl\||u-v|(|\grad u| + |\grad v|)(|f''(u)| + |f''(v)|)\bigr\|_{L^2} \\
      \lesssim\,&\|u-v\|_{L^{10}}\cdot(\|\grad u\|_{L^{10/3}}+\|\grad v\|_{L^{10/3}})\cdot\big(\|f''(u)\|_{L^{10}} + \|f''(v)\|_{L^{10}}\big) \\
      \lesssim\,&\|u-v\|_{X^1}\cdot\big(f'(\|u\|_{X^1}) + f'(\|v\|_{X^1})\big).
    \end{aligned}
  \end{equation*}
\end{proof}

\begin{proof}[Proof of Proposition 
\ref{prop:energy-est}]
  Let $B$ denote the ball of centre $0$ and radius $2\|(u_0, u_1)\|_{X^1\times H^1}$ in the space $X^1 \times H^1$.
  Given $(u, \partial_t u)\in C([0, T], B)$, let $\wt u = \Phi(u)$ denote the solution of the Cauchy problem
  \begin{equation*}
    \begin{cases}
      \partial_{tt}\wt u - \Delta \wt u = f(u), \\
      (\wt u(0), \partial_t \wt u(0)) = (u_0, u_1)
    \end{cases}
  \end{equation*}
  It follows from Lemma \eqref{eq:contr1} and \eqref{eq:energy-lin} that if $T \leq \frac{M_0}{C f(2M_0)}$, then $(\wt u, \partial_t \wt u) \in C([0, T], B)$.
  It follows from \eqref{eq:contr2} and \eqref{eq:energy-lin} that if $T \leq \frac{1}{4Cf'(2M_0)}$, then $\Phi$ is a contraction,
  so it has a unique fixed point, which is the desired solution.

  The function $v := u-u\lin$ solves the Cauchy problem
  \begin{equation*}
    \begin{cases}
      \partial_{tt}v - \Delta v = f(u), \\
      (v(0), \partial_t v(0)) = 0,
    \end{cases}
  \end{equation*}
  so \eqref{eq:energy-nonlin-2} follows from \eqref{eq:energy-lin}.
\end{proof}

\subsection{Persistence of $X^1 \times H^1$ regularity}
We recall the classical Strichartz inequality:
\begin{lemma}\cite{GiVe95}
  \label{lem:strichartz}
  Let $I = [0, T_0]$ be a time interval, $g\in C(I, L^2)$ and $(u_0, u_1) \in \enorm$. Let $\bs u$ be the solution of the Cauchy problem
  \begin{equation*}
    \begin{cases}
      \partial_{tt}u - \Delta u = g, \\
      (u(0), \partial_t u(0)) = (u_0, u_1).
    \end{cases}
  \end{equation*}
  Then
  \begin{equation*}
    \|u\|_{L^{7/3}(I; L^{14/3})} \lesssim \|(u_0, u_1)\|_\enorm + \|g\|_{L^1(I; L^2)},
  \end{equation*}
  with a constant independent of $I$.
\end{lemma}
From the local theory of \eqref{eq:nlw} in the critical space we know that if $\bs u \in C((T_-, T_+); \dot H^1\times L^2)$ is a solution
of \eqref{eq:nlw} and $I = [T_1, T_2] \subset (T_-, T_+)$, then
\begin{equation}
  \label{eq:str-bound}
  \|u\|_{L^{7/3}(I; L^{14/3})} < +\infty.
\end{equation}
\begin{proposition}
  \label{prop:persistence}
  Suppose that $0 \in I = [T_1, T_2]\subset (T_-, T_+)$ and that $(u_0, u_1) \in X^1\times H^1$.
  Then $\bs u \in C(I, X^1\times H^1)$.
\end{proposition}
\begin{proof}
  The proof is classical, see for example \cite[Chapter 5]{Cazenave03} for more general results in the case of NLS.

  We consider positive times. The proof for negative times is the same.
  Let $T_*$ be the maximal time of existence of $\bs u$ in $X^1 \times H^1$.
  Suppose that $T_* < T_+$. From Proposition 
\ref{prop:energy-est} it follows that
  \begin{equation}
    \label{eq:blowup-H2}
    \lim_{t \to T_*}\|\bs u\|_{X^1\times H^1} = +\infty.
  \end{equation}
  Consider the time interval $I = [T_* - \tau, T_*]$. Derivating \eqref{eq:nlw} once and using Lemma 
\ref{lem:strichartz} we get
  \begin{equation}
    \label{eq:extrap}
    \|\grad u\|_{L^{7/3}(I; L^{14/3})} \leq C \|(u(T_*-\tau), \partial_t u(T_* - \tau))\|_{X^1 \times H^1} + C\|\grad(f(u))\|_{L^1(I; L^2)},
  \end{equation}
  with $C$ independent of $\tau$. From H\"older inequality we have
  \begin{equation*}
    \|\grad(f(u))\|_{L^1(I; L^2)} \leq \|\grad u\|_{L^{7/3}(I; L^{14/3})}\cdot f'\big(\|u\|_{L^{7/3}(I; L^{14/3})}\big).
  \end{equation*}
  By \eqref{eq:str-bound}, the last term is arbitrarily small when $\tau \to 0^+$, so for $\tau$ small enough the second term on the right hand side of \eqref{eq:extrap}
  can be absorbed by the left hand side, which implies $\|\grad u\|_{L^{7/3}(I; L^{14/3})} < +\infty$ and $\|\grad(f(u))\|_{L^1(I; L^2)} < +\infty$.
  This is in contradiction with \eqref{eq:blowup-H2}, because of the energy estimate~\eqref{eq:energy-lin}.
\end{proof}

\subsection{Propagation of regularity around a non-degenerate point}
\begin{proposition}
  \label{prop:cauchy-nondeg}
  Let $(u_0, u_1) \in X^4 \times H^4$ such that $u_0(0) > 0$. Let $(u, \partial_t u) \in C([0, T_0]; X^1\times H^1)$ be the solution of the Cauchy problem:
  \begin{equation*}
    \begin{cases}
      \partial_{tt}u - \Delta u = f(u), \\
      (u(0), \partial_t u(0)) = (u_0, u_1),
    \end{cases}
  \end{equation*}
  constructed in Proposition 
\ref{prop:energy-est}. There exists $\tau, \rho > 0$ such that $(u, \partial_t u)$ satisfies:
  \begin{equation}
    \label{eq:cut-reg}
    \big(\chi\big(\frac{\cdot}{\rho}\big) u, \chi\big(\frac{\cdot}{\rho}\big)\partial_t u\big) \in C([0, \tau]; X^4\times H^4)
  \end{equation}
  (where $\chi$ is a standard regular cut-off function).
\end{proposition}
\begin{proof}
  Denote $v_0 := u_0(0) > 0$
  and introduce an auxiliary function $\wt f \in C^\infty$, $\wt f(u) = f(u)$ when $u \geq v_0 / 2$, $f(u) = 0$ when $u \leq 0$.
  Using Fa\`a di Bruno formula one can prove an analog of Lemma 
\ref{lem:energy-est}:
  \begin{align*}
    \|\wt f(u)\|_{H^4} &\leq C(\|u\|_{X^4}), \\
    \|\wt f(u) - \wt f(v)\|_{H^4} &\leq \|u-v\|_{X^4}\cdot C(\|u\|_{X^4} + \|v\|_{X^4}),
  \end{align*}
  where $C: \bR_+ \to \bR_+$ is a continuous function. The same procedure as in the proof of Proposition 
\ref{prop:energy-est}
  leads to the conlusion that there exists $\tau > 0$ such that the Cauchy problem:
  \begin{equation*}
    \begin{cases}
      \partial_{tt}\wt u - \Delta \wt u = \wt f(\wt u), \\
      (\wt u(0), \partial_t \wt u(0)) = (u_0, u_1)
    \end{cases}
  \end{equation*}
  has a solution $(\wt u, \partial_t \wt u) \in C([0, \tau], X^4 \times H^4)$.
  By continuity and Schauder estimates, if we take $\tau$ and $\rho$ sufficiently small, we have $\wt u(t, x) > \frac 12 v_0$ for $|x| \leq 4\rho$ and $0\leq t\leq \tau$.
  We may assume that $\tau \leq 2\rho$.
  Consider $v = u - \wt u$. We will prove that $v = 0$ when $0 \leq t \leq \tau$ and $|x| \leq 2\rho$, which will finish the proof.
  The function $v$ solves the Cauchy problem:
  \begin{equation*}
    \begin{cases}
      \partial_{tt}v - \Delta v = f(u) - \wt f(\wt u), \\
      (v(0), \partial_t v(0)) = 0.
    \end{cases}
  \end{equation*}
  We run the localized energy estimate \eqref{eq:energy-lin-loc} for $|x| \leq 2\rho + |t - \tau|$. We suppose that $\tau \leq 2\rho$, so $|x| \leq 4\rho$,
  which means that $\|f(u) - \wt f(\wt u)\|_{H^1} = \|f(u) - f(\wt u)\|_{H^1} \lesssim \|u - \wt u\|_{X^1}$ (the norm is taken in the ball $B(0, 2\rho + |t - \tau|$).
  From \eqref{eq:energy-lin-loc} and Gronwall inequality we deduce that $u = \wt u$ when $|x| \leq 2\rho + |t - \tau|$, in particular when $|x| \leq 2\rho$. 
\end{proof}

\subsection{Short-time asymptotics in the case $(u_0, u_1) = (p|x|^\beta, 0)$.}
Let $(u, \partial_t u)$ denote the solution of \eqref{eq:nlw} corresponding to the intial data
\begin{equation*}
  (u_0, u_1) = \big(\chi\big(\frac{\cdot}{\rho}\big)p|x|^\beta, 0\big),
\end{equation*}
where $p, \rho > 0$ and $\beta > \frac 52$ are constants and $\chi$ is a standard cut-off function.
Let $(u\lin, \partial_t u\lin)$ denote the solution of the free wave equation corresponding to the same initial data.
\begin{proposition}
  \label{prop:u-lin}
  Let $q = \frac{(\beta+1)(\beta+3)}{3}p$. There exist $T_0 > 0$ and a constant $C > 0$ such that
  for $0 \leq t \leq T_0$ and $|x| \leq \frac 12 t$ there holds
  \begin{equation}
    \label{eq:u-lin}
    |u\lin(t, x) - qt^\beta| \leq Ct^{\beta-2}|x|^2.
  \end{equation}
\end{proposition}
\begin{proof}
  Define
  \begin{equation*}
    w(y) := \fint_{\partial B(0, 1)} p|\omega + ye_1|^\beta\ud\sigma(\omega),\qquad -\frac{1}{\sqrt 2} < y < \frac{1}{\sqrt 2},
  \end{equation*}
  where $B(0, 1)$ denote the unit ball in $\bR^5$, $\ud\sigma$ is the surface measure on the unit sphere and $e_1 = (1, 0, 0, 0, 0)$.
  Notice that
  \begin{equation*}
    |\omega + ye_1|^\beta = (1-\omega_1^2 + (y + \omega_1)^2)^{\beta/2} = (1+\omega_1^2)^{\beta/2}\cdot\big(1 + y\frac{2\omega_1}{1+\omega_1^2}\big)^{\beta/2}
  \end{equation*}
  can be developped in a power series in $y$ which converges uniformly for $-\frac{1}{\sqrt 2} < y < \frac{1}{\sqrt 2}$.
  Hence, $w$ is an analytic function. It is also symetric, so it is in fact analytic in $y^2$, $$w(y) = \wt w(y^2),\qquad \wt w(z) \text{ analytic for }|z| < \frac 12.$$
  We have $\wt w(0) = w(0) = p$.

  The representation formula for solutions of the free wave equation, see for example \cite[p. 77]{evans98}, yields
  \begin{equation*}
    u\lin(t, x) = \frac 13 \big(\pd t\big)\big(\frac 1t \pd t\big)\big(t^3 \fint_{\partial B(x, t)}p|y|^\beta\ud \sigma(y)\big).
  \end{equation*}
  A change of variables shows that for $|x| < \frac{1}{\sqrt 2}t$ and $t$ sufficiently small we have
  \begin{equation*}
    u\lin(t, x) = \frac 13 \big(\pd t\big)\big(\frac 1t \pd t\big)\big(t^3 \cdot t^\beta \wt w\big(\frac{|x|^2}{t^2}\big)\big) = t^\beta \wt w_1\big(\frac{|x|^2}{t^2}\big),
  \end{equation*}
  where $\wt w_1(z)$ is analytic for $|z| < \frac 12$. It is easily seen that $\wt w_1(0) = \frac{(\beta+1)(\beta+3)}{3}p = q$
  (all the terms coming from differentiating $\wt w$ vanish at $z = 0$).
  Hence, there exists a constant $C$ such that $|\wt w_1(z) - q| \leq C |z|$ for $|z| \leq \frac 14$, and the conclusion follows.
\end{proof}

\begin{proposition}
  \label{prop:u-nonlin}
  For $t$ small enough there holds
  \begin{equation*}
    \|u-u\lin\|_{X^1(|x|\leq \frac 12 t)} \lesssim t^{\frac 73 \beta + \frac 76}.
  \end{equation*}
\end{proposition}
\begin{proof}
  From \eqref{eq:energy-nonlin-2} and finite speed of propagation we obtain
  \begin{equation*}
    \|u-u\lin\|_{X^1(|x|\leq \frac 12 t)} \lesssim f(\|(u_0, u_1)\|)_{X^1\times H^1(|x| \leq \frac 32 t)}.
  \end{equation*}
  We have
  \begin{equation*}
    \|(u_0, u_1)\|_{X^1\times H^1(|x|\leq \frac 32 t)}^2 \sim \int_0^{\frac 32 t} (r^{\beta-2})^2 r^4 \ud r \sim t^{2\beta+1},
  \end{equation*}
  and the conclusion follows.
\end{proof}

\bibliographystyle{plain}
\bibliography{construction}

\providecommand{\noopsort}[1]{}
\begin{thebibliography}{10}

\bibitem{BaChDa11}
H.~Bahouri, J.-Y. Chemin, and Danchin R.
\newblock {\em {Fourier Analysis and Nonlinear Partial Differential
  Equations}}, volume 343 of {\em Grundlehren der mathematischen
  Wissenschaften}.
\newblock Springer, 2011.

\bibitem{BaGe99}
H.~Bahouri and P.~G{\'e}rard.
\newblock High frequency approximation of solutions to critical nonlinear wave
  equations.
\newblock {\em Amer. J. Math.}, 121(1):131--175, 1999.

\bibitem{BoWa97}
J.~Bourgain and W.~Wang.
\newblock Construction of blowup solutions for the nonlinear {S}chr{\"o}dinger
  equation with critical nonlinearity.
\newblock {\em Ann. Scuola Norm. Sup. Pisa Cl. Sci. (4)}, 25:197--215, 1997.

\bibitem{Cazenave03}
T.~Cazenave.
\newblock {\em {Semilinear Schr{\"o}dinger Equations}}, volume~10 of {\em
  Courant Lecture Notes in Mathematics}.
\newblock AMS, 2003.

\bibitem{Collot14p}
C.~Collot.
\newblock Type {II} blow up for the energy supercritical wave equation.
\newblock {\em to appear in Mem. Amer. Math. Soc.}, arXiv:1407.4525, 2014.

\bibitem{DKM1}
T.~Duyckaerts, C.~E. Kenig, and F.~Merle.
\newblock Universality of blow-up profile for small radial type {II} blow-up
  solutions of the energy-critical wave equation.
\newblock {\em J. Eur. Math. Soc.}, 13(3):533--599, 2011.

\bibitem{DKM2}
T.~Duyckaerts, C.~E. Kenig, and F.~Merle.
\newblock Universality of the blow-up profile for small type {II} blow-up
  solutions of the energy-critical wave equation: the nonradial case.
\newblock {\em J. Eur. Math. Soc.}, 14(5):1389--1454, 2012.

\bibitem{DM08}
T.~Duyckaerts and F.~Merle.
\newblock Dynamics of threshold solutions for energy-critical wave equation.
\newblock {\em Int. Math. Res. Pap. IMRP}, 2008.

\bibitem{evans98}
L.~C. Evans.
\newblock {\em {Partial Differential Equations}}, volume~19 of {\em Graduate
  Studies in Mathematics}.
\newblock AMS, 1998.

\bibitem{GiVe95}
J.~Ginibre and G.~Velo.
\newblock Generalized {S}trichartz inequalities for the wave equation.
\newblock {\em J. Funct. Anal.}, 133:50--68, 1995.

\bibitem{HiRa12}
M.~Hillairet and P.~Rapha{\"e}l.
\newblock Smooth type {II} blow up solutions to the four dimensional energy
  critical wave equation.
\newblock {\em Anal. PDE}, 5(4):777--829, 2012.

\bibitem{KeMe08}
C.~E. Kenig and F.~Merle.
\newblock Global well-posedness, scattering and blow-up for the energy-critical
  focusing non-linear wave equation.
\newblock {\em Acta Math.}, 201(2):147--212, 2008.

\bibitem{KrNaSc15}
J.~Krieger, K.~Nakanishi, and W.~Schlag.
\newblock Center-stable manifold of the ground state in the energy space for
  the critical wave equation.
\newblock {\em Math. Ann.}, 361(1--2):1--50, 2015.

\bibitem{KrSc14}
J.~Krieger and W.~Schlag.
\newblock Full range of blow up exponents for the quintic wave equation in
  three dimensions.
\newblock {\em J. Math Pures Appl.}, 101(6):873--900, 2014.

\bibitem{KrScTa08}
J.~Krieger, W.~Schlag, and D.~Tataru.
\newblock Renormalization and blow up for charge one equivariant critical wave
  maps.
\newblock {\em Invent. Math.}, 171(3):543--615, 2008.

\bibitem{KrScTa09}
J.~Krieger, W.~Schlag, and D.~Tataru.
\newblock Slow blow-up solutions for the {$H\sp 1(\mathbb{R}\sp 3)$} critical
  focusing semilinear wave equation.
\newblock {\em Duke Math. J.}, 147(1):1--53, 2009.

\bibitem{Martel05}
Y.~Martel.
\newblock Asymptotic ${N}$-soliton-like solutions of the subcritical and
  critical generalized {K}orteweg-de {V}ries equations.
\newblock {\em Amer. J. Math.}, 127(5):1103--1140, 2005.

\bibitem{MMR15-3}
Y.~Martel, F.~Merle, and P.~Rapha{\"e}l.
\newblock Blow up for the critical {gKdV} equation {III}: exotic regimes.
\newblock {\em Ann. Sc. Norm. Super. Pisa Cl. Sci.}, XIV:575--631, 2015.

\bibitem{MMRS14}
Y.~Martel, F.~Merle, P.~Rapha{\"e}l, and J.~Szeftel.
\newblock Near soliton dynamics and singularity formation for ${L}^2$ critical
  problems.
\newblock {\em Russ. Math. Surv.}, 69(2):261--290, 2014.

\bibitem{Merle90}
F.~Merle.
\newblock Construction of solutions with exactly $k$ blow-up points for the
  {S}chr{\"o}dinger equation with critical nonlinearity.
\newblock {\em Commun. Math. Phys.}, 129(2):223--240, 1990.

\bibitem{OrPe13}
C.~Ortoleva and G.~Perelman.
\newblock Nondispersive vanishing and blow up at infinity for the energy
  critical nonlinear {S}chr{\"o}dinger equation in $\mathbb{R}^3$.
\newblock {\em Algebra i Analiz}, 25(2):162--192, 2013.

\bibitem{Perelman14}
G.~Perelman.
\newblock Blow up dynamics for equivariant critical {S}chr{\"o}dinger maps.
\newblock {\em Commun. Math. Phys.}, 330(1):69--105, 2014.

\bibitem{RaSz11}
P.~Rapha{\"e}l and J.~Szeftel.
\newblock Existence and uniqueness of minimal mass blow up solutions to an
  inhomogeneous ${L}^2$-critical {NLS}.
\newblock {\em J. Amer. Math. Soc.}, 24(2):471--546, 2011.

\bibitem{Strauss77}
W.~A. Strauss.
\newblock Existence of solitary waves in higher dimensions.
\newblock {\em Comm. Math. Phys.}, 55:149--162, 1977.

\bibitem{teschl12}
G.~Teschl.
\newblock {\em {Ordinary Differential Equations and Dynamical Systems}}, volume
  140 of {\em Graduate Studies in Mathematics}.
\newblock AMS, 2012.

\end{thebibliography}

\end{document}